\documentclass[10pt]{article}
\usepackage{amsmath, amsthm, amssymb, pdfsync, psfrag, verbatim,color}

\usepackage{hyperref}
\IfFileExists{bbm.sty}{\usepackage{bbm}}{}
\usepackage {mathrsfs, graphicx,wrapfig,multirow}
\usepackage[round]{natbib}

\usepackage{enumitem}
\newcommand{\mylabel}[2]{#2\def\@currentlabel{#2}\label{#1}}

\IfFileExists{breakcites.sty}{\usepackage{breakcites}}{}
\usepackage[affil-it]{authblk}

\IfFileExists{ulem.sty}{\usepackage[normalem]{ulem}}{}
\usepackage{xcolor}

\newcommand{\Hzero}{(\mathrm{H}_0)}
\newcommand{\Hqunif}{(\mathrm{H}_{q}^{\mathrm{unif}})}
\newcommand{\Hqsep}{(\mathrm{H}_{q}^{\mathrm{sep}})}

\newcommand{\Hlip}{(\mathrm{H}_{\mathrm{lip}})}
\newcommand{\Hbound}{(\mathrm H_q^{bd})}

\def\E{{\mathbb E}}

\newcommand{\F}{\mathcal{F}}
\newcommand{\prob}{\mathbb{P}}

\newcommand{\bbE}{\mathbb E}

\newcommand{\bbH}{\mathbb H}

\newcommand{\bbL}{\mathbb L}

\newcommand{\bbN}{\mathbb N}

\newcommand{\bbP}{\mathbb P}
\newcommand{\bbQ}{\mathbb Q}
\newcommand{\bbR}{\mathbb R}
\newcommand{\bbS}{\mathbb S}

\newcommand{\scA}{\mathcal A}
\newcommand{\scB}{\mathcal B}

\newcommand{\scE}{\mathcal E}
\newcommand{\scF}{\mathcal F}

\newcommand{\scT}{\mathcal T}

\newcommand{\veps}{\varepsilon}

\newcommand{\norm}[1]{\ensuremath{\left\| #1 \right\|}}
\newcommand{\abs}[1]{\ensuremath{\left| #1 \right|}}

\DeclareMathOperator*{\esssup}{ess\,sup}

\newcommand{\diff}{\mathop{}\!d}

\newcommand{\indicator}[1]{\ensuremath{\mathbf{1}_{#1}}}

\newcommand{\half}{\frac{1}{2}}

\newcommand{\crl}[1]{\ensuremath{ \left\{ #1 \right\} }}
\newcommand{\edg}[1]{\ensuremath{ \left[ #1 \right] }}
\newcommand{\brak}[1]{\ensuremath{\left( #1 \right)}}

\newcommand{\dis}{\text{dist}}

\renewcommand{\abstractname}{Abstract}
\renewenvironment{abstract}
  {\small\begin{quotation}
   \noindent{\normalfont\bfseries\abstractname.}\space\ignorespaces}
  {\end{quotation}}
\makeatletter
\g@addto@macro\th@plain{%
  \thm@headfont{\normalfont\bfseries}%
  \thm@notefont{\normalfont\bfseries}%
}
\makeatother
\theoremstyle{plain}

\newtheorem{theorem}{Theorem}[section]
\newtheorem{definition}[theorem]{Definition}
\newtheorem{proposition}[theorem]{Proposition}
\newtheorem{corollary}[theorem]{Corollary}
\newtheorem{lemma}[theorem]{Lemma}
\newtheorem{remark}[theorem]{Remark}
\newtheorem{example}[theorem]{Example}
\newtheorem{examples}[theorem]{Examples}
\newtheorem{foo}[theorem]{Remarks}

\newenvironment{Remark}{\begin{remark}\rm}{\end{remark}}

\title{\textbf{Quadratic Mean-Field BSDEs and Exponential Utility Maximization}\thanks{This work is supported by the National Natural Science Foundation of China (Grant No. 12571478), the Guangdong Basic and Applied Basic Research Foundation (Grant No. 2025B1515020091), and the Shenzhen Fundamental Research General Program (Grant No. JCYJ20230807093309021). \\
		Emails: y.ding@maths.usyd.edu.au, kihun.nam@monash.edu, wenjq@sustech.edu.cn.}
}

\author[a]{Yining Ding}
\author[b]{Kihun Nam}
\author[c]{Jiaqiang Wen}

\affil[a]{School of Mathematics and Statistics, The University of Sydney}
\affil[b]{School of Mathematics, Monash University}
\affil[c]{Department of Mathematics, Southern University of Science and Technology}
\date{ }
\begin{document}
	\maketitle
	\begin{abstract}
		In this paper, we study a class of real-valued mean-field backward stochastic differential equations (BSDEs) with generators of quadratic growth in the control variable and the mean-field term. For bounded terminal conditions satisfying suitable structural regularity, we establish existence and uniqueness under two complementary quadratic regimes; the uniformly continuous path-dependent case additionally requires a uniform quadratic bound on the normalized generator. Our approach for these results combines Malliavin calculus with refined BMO and stability estimates, and bridges the gap between the quadratic BSDE results of [Ann. Probab. 45 (2017), pp. 3795--3828] and [Ann. Appl. Probab. 35 (2025), pp. 2128--2174]. Motivated by a mean-field exponential utility maximization problem, we further extend our framework to terminal conditions  without imposing continuity or Markovian structure on the terminal condition. Under a suitable smallness condition on the oscillation of the terminal value around its mean, the associated quadratic BSDE admits a unique solution in a suitable small solution class. We then apply the extended theory to solve a mean-field exponential utility maximization problem, thereby extending the constrained exponential utility framework of Hu et al. [Ann. Appl. Probab. 15 (2005), pp. 1691--1712] to a mean-field setting.
		\\[2mm]
		\textbf{Key words:} Mean-field BSDEs; quadratic BSDEs; BMO martingales; Malliavin calculus; exponential utility; maximal utility.\\[1mm]
		\noindent\textbf{2020 Mathematics Subject Classification:}
		\rm   Primary 60H30;  Secondary 60H07; 60H10.

	\end{abstract}

	\setcounter{equation}{0}
	
	\section{Introduction}
	Let $(\Omega,\scF,\{\scF_t\}_{t \geq 0},\bbP)$ be a filtered probability space where $W$ is a $d$-dimensional Brownian motion and $\{\scF_t\}_{t \geq 0}$ is the Brownian filtration. Consider the following mean-field BSDE:
	\begin{align}\label{intro:mfbsde}
		Y_t= \xi+\int_t^T \int_{\tilde{\Omega}} g(s,Z_s,\tilde Z_s) \diff \tilde\bbP ds-\int_t^TZ_s  \!\cdot dW_s,
	\end{align}
	where the process $\tilde Z$ is an independent copy of $Z$ defined on a duplicated filtered probability space $(\tilde\Omega,\tilde \scF,\{\tilde{\scF}_t\}_{t \geq 0},\tilde\bbP)$. Here, $\xi$ is $\scF_T$-measurable bounded random variable and $g:[0,T]\times\bbR^d\times\bbR^d\to\bbR$ is a deterministic measurable map. In this article, we want to prove the existence and the uniqueness of an adapted couple solution $(Y,Z)$ such that $Y$ is bounded and $Z \cdot W$ is a BMO (bounded in mean oscillation) martingale when $g(s,z,z')$ has a quadratic growth in both $(z,z')$. 
	
	The origin of linear BSDEs can be traced back to their role as adjoint equations in stochastic control problems, first studied by \cite{bismut1973conjugate}. The seminal paper by \cite{pardoux1990adapted} established the existence and uniqueness of adapted solutions $(Y, Z)$ for 
	\begin{align}
		Y_t = \xi + \int_t^T F(s, Y_s, Z_s) \, ds - \int_t^T Z_s  \!\cdot dW_s, \quad t \in [0, T],
	\end{align}
	under the conditions that the generator $F$ is Lipschitz continuous and the terminal condition $\xi$ is square integrable. Since then, there has been growing interest in BSDEs, with applications in stochastic control and differential games \citep{bismut1976linearquad,sun2023zerosum,wen2023stochastic}, mathematical finance, and partial differential equations. For a comprehensive presentation of the theory, see \cite{el1997backward}, \cite{CohenStocal2015}, or \cite{ZhangBSDEbook2017}. Numerous studies have aimed to improve existence and uniqueness results by relaxing the Lipschitz condition on the generator $F$ or the conditions on the terminal condition $\xi$. 
	
	BSDEs with generators that grow quadratically in the variable $z$ have attracted significant attention. These equations arise, for example, in stochastic linear-quadratic control with random coefficients \cite{bismut1976linearquad}, in utility optimization problems in an incomplete market \cite{hu2005utility}. \cite{Kobylanski:2000cy} provided the seminal existence-and-uniqueness result for real-valued quadratic BSDEs with bounded terminal values, in which $F$ has quadratic growth in its $Z$ component. \cite{briand2013simple} have continued to explore this area, demonstrating the existence of unique solutions to quadratic BSDEs with bounded terminal conditions using purely probabilistic methods. Investigations by \cite{briand2006bsde} extend the previous results to unbounded terminal conditions. Some studies consider multidimensional cases; see \cite{hu2016multi,hao2022solvability,tang2023MultiMFBSDE,luo2019triangle,tevzadze2008solvability}.
	
	Mean-field stochastic differential equations, also known as McKean-Vlasov equations, describe stochastic systems whose evolution is influenced by both the microscopic state and the macroscopic distribution of particles. These equations date back to \cite{kac1956foundations}, which presented work on the Boltzmann and Vlasov equations in the 1950s. Mean-field game theory began with the groundbreaking contributions of \cite{lasry2007mfgame}. Building on this foundation, \cite{Buckdahn2009MFlimit} introduced mean-field BSDEs, derived from approximations involving weakly interacting $ N$- particle systems, and investigated them using probabilistic methods. This has led to a growing interest in mean-field BSDEs and FBSDEs. For instance, \cite{Buckdahn2009MFBSDE} studied the existence and uniqueness of solutions with Lipschitz-continuous generators and square-integrable terminal conditions. More recently, \cite{hao2022mean} explored the solvability of solutions for one-dimensional mean-field BSDEs with quadratic growth. In contrast, \cite{tang2023MultiMFBSDE} established the existence and uniqueness of solutions to multidimensional BSDEs with generators that depend on the expectation of both variables. \cite{Li2018MFBSDEcontinuous} examined mean-field BSDEs whose coefficients depend not only on the solution $(Y, Z)$ but also on the law of $Y$. \cite{BuckdahnLiLi2026} derived a global stochastic maximum principle for mean-field forward–backward stochastic control problems with quadratic BSDE generators via BMO-based analysis. 
	
	\cite{el1997backward} and \cite{Buckdahn2009MFBSDE} state that, in a Clark–Ocone type formula, the martingale density component $Z$ of the solution pair of a BSDE with a classical globally Lipschitz generator on a Gaussian basis is the Malliavin derivative of the other component $Y$ (see Proposition 5.3). The result has been extending our understanding of quadratic BSDEs: for example, \cite{briand2013simple} provided a straightforward approach to the existence and uniqueness of quadratic BSDEs with bounded terminal conditions using Malliavin calculus to overcome difficulties with quadratic generators, and \cite{Imkeller2010path} proposed a numerical scheme for BSDEs with generators of quadratic growth.
	
	Continuous-time portfolio optimization under the utility-maximization framework introduced by \cite{merton1969portfolio} is a foundational topic in mathematical finance. Merton analyzed the problem in a complete market where the risky asset follows the Black-Scholes dynamics and no constraints are imposed on admissible trading strategies. A major development in the early 2000s was the emergence of a connection between this class of optimization problems and the theory of BSDEs, which has since become a central approach in the field. The core idea of the BSDE methodology is to construct a stochastic process, depending on the investment strategy, that matches the investor’s utility of terminal wealth. In this context, \cite{elrouge2000utility} used BSDEs to compute the value function and the optimal strategy for an exponential utility, assuming that trading strategies are restricted to a convex cone. Building on this framework, \cite{hu2005utility} developed a BSDE approach {(martingale duality method)} to utility maximization in incomplete markets and extended the methodology to include logarithmic and power utility functions, allowing for general strategy constraints within a closed set. The application developed in Section \ref{Sec: application} retains this BSDE approach with an endogenous mean-field  condition.

	Our methods are inspired by the works of \cite{tevzadze2008solvability} and \cite{briand2013simple}. The former studied the existence and uniqueness of solutions to multidimensional quadratic BSDEs using BMO martingale theory. The latter provided a straightforward approach to establishing the existence and uniqueness of quadratic BSDEs with bounded terminal conditions using Malliavin calculus. Building on these foundations, we bridge the gap between the approaches of \cite{cheridito2017bse} and \cite{hao2022mean}. Specifically, \cite{hao2022mean} established the existence and uniqueness by applying a contraction mapping with the Reverse H\"older Inequality (RHI) and iterating backward in time when the generator is quadratic in the $Z$ component, while \cite{cheridito2017bse} used the Krasnoselskii fixed-point theorem on a small time interval, iterating backward to obtain a global solution. Our contribution concerns nonadditive interactions between $Z$ and the mean-field component $\tilde Z$, under the terminal regularity and generator assumptions stated below. Additive quadratic mean-field terms are already covered by \cite{hao2022mean}.
	
	Our main contributions are as follows. First, we prove existence and uniqueness of quadratic mean-field BSDEs when the generator satisfies the \emph{unified} quadratic condition, that is, quadratic in both $Z$ and its mean-field component $\tilde Z$. In this fully coupled regime, the generator may contain genuine interactions between the control variable \(Z\) and the mean-field component \(\tilde Z\). To the best of our knowledge, this is the first paper that studies such a case in a general framework. Under this condition, we prove well-posedness for bounded terminal conditions of the form
	\(\xi=\phi(X_{t_1},X_{t_2},\cdots,X_{t_n})\), $t_i\in[0,T]$, where \(\phi\) is \(\alpha\)-H\"older continuous for $\alpha\in(0,1]$ and $X$ is a solution of a forward stochastic differential equation. Second, under the \emph{separable} quadratic condition and the additional normalized-growth bound $\Hbound$, we obtain existence and uniqueness of the form \(\xi=\phi(W_{[0,T]})\), where \(\phi:C([0,T];\bbR^d)\to\bbR\) is uniformly continuous in the sup norm. Third, we extend the constrained exponential utility maximization framework of \cite{hu2005utility} to a mean-field setting in which the feasible portfolio set depends endogenously on the aggregated BSDE signal. This feedback leads to an expectation-dependent quadratic BSDE. Without imposing the structural regularity required in the first two results, we show that the corresponding quadratic BSDE admits a unique solution in a suitable small solution class under a small-terminal-oscillation condition, obtained through a local contraction argument and backward concatenation. The resulting BSDE characterizes the value function and an optimal trading strategy through projection onto the mean-field adaptive constraint set. For the first two well-posedness results, our approach builds on BMO martingale theory, Malliavin calculus, and stability estimates to address the analytical difficulties arising from the quadratic structure, rather than relying on a fixed-point construction. A key novelty throughout is the inclusion of a quadratic mean-field component \(\tilde Z\), which substantially alters the structure compared to classical quadratic mean-field (MF) BSDEs.
	
	The rest of the paper is organized as follows. In Section \ref{Sec: Mainresult}, we present the framework, introduce the standing assumptions (including $\Hqunif$ and $\Hqsep$), and state our main results. Section \ref{Sec: MalliavinD} derives an almost-sure upper bound for the solution pair \( (Y, Z) \) of Lipschitz mean-field BSDEs, assuming that the terminal condition admits a uniformly bounded Malliavin derivative, as shown in Proposition \ref{prop: lipbound}. Importantly, the resulting bound is independent of the Lipschitz constants in \( z \) and \( \tilde{z} \). In the main Section \ref{Sec: Stability}, we provide a priori estimates and establish a stability result tailored to mean-field BSDEs with quadratic growth, which will be used to pass from Lipschitz truncations to quadratic generators. In Section \ref{Sec: ExistenceUniqueness}, we combine the uniform Lipschitz bounds obtained in Section \ref{Sec: MalliavinD} with the stability result and a density argument to prove existence and uniqueness under the stated terminal assumptions, using $\Hbound$ in the uniformly continuous case. Finally, Section \ref{Sec: application} extends the constrained exponential utility framework of \cite{hu2005utility} to endogenous mean-field adaptive constraints. We derive the associated expectation-dependent BSDE, establish its well-posedness under a small-terminal-oscillation condition, and use it to characterize the value function and an optimal trading strategy.
	
	\section{Preliminaries and statement of main results} \label{Sec: Mainresult}
	\subsection{Preliminaries}
	For a given finite maturity $T$ and a filtered probability space $(\Omega, \scF, (\scF_t)_{t\in[0, T]}, \bbP)$ where the filtration is the usual augmentation of the natural filtration of a $d$-dimensional standard Brownian motion $W=(W_t)_{t\geq0}$. Then all martingales admit an RCLL modification (i.e., right-continuous with left limits). Let $(\tilde{\Omega}, \tilde{\scF}, (\tilde{\scF}_t)_{t\in[0,T]}, \tilde{\bbP})$ be a copy of the probability space $(\Omega, \scF, (\scF_t)_{t\in[0,T]}, \bbP)$. For any random variable (of arbitrary dimension) $X$ over $(\Omega, \scF, \bbP)$ we denote by $\tilde{X}$ a copy (i.e., of the same law as $X$), but defined over $(\tilde{\Omega}, \tilde{\scF}, \tilde{\bbP})$. The expectation $\tilde{\bbE}$ acts only over the variables endowed with a tilde. By $\abs{\cdot}$, we denote the Euclidean norm on $\bbR^d$, and for $p > 1$, we denote by
	\begin{itemize}
		\item $L^p$: all real-valued random variables $X$ s.t. $\|X\|_{L^p} := \left(\E^\prob\big[ |X|^p \big]\right)^{\frac{1}{p}} < \infty$,
		\item $L^\infty$: all essentially bounded random variables $X$ s.t.
		$ \|X\|_{L^\infty} := \esssup_{\omega \in \Omega} |X(\omega)| <\infty$,
		\item $\bbS^p$: the set of real valued $\scF$-adapted continuous processes $(Y_t)_{0\leq t \leq T}$ satisfying $$\norm{Y}_{\bbS^p} := \left(\bbE\left[\sup_{0\leq r \leq T} \abs{Y_r}^p\right]\right)^{\frac{1}{p}} < \infty,$$
		\item $\bbS^\infty$: the set of real valued $\scF$-adapted continuous processes $(Y_t)_{0\leq t \leq T}$ satisfying
		$$\norm{Y}_{\bbS^\infty} := \sup_{0 \leq r \leq T} \norm{Y_r}_{L^\infty} = \esssup_{\omega\in \Omega}\sup_{0\leq r\leq T} \abs{Y_r(\omega)} < \infty,$$
		\item $\bbH^p$: the set of predictable $\bbR^d$-valued processes $Z$ satisfying
\begin{align*}
\norm{Z}_{\bbH^p}&:=\left(\bbE\left[\left(\int_0^T\abs{Z_r}^2\,dr\right)^{p/2}\right]\right)^{1/p}<\infty\qquad\text{ for } p<\infty, \text{ and}\\
\|Z\|_{\mathbb H^\infty}&:=
	\mathop{\mathrm{ess\,sup}}_{(t,\omega)\in[0,T]\times\Omega}|Z_t(\omega)|<\infty.
\end{align*}
		\item $\bbH^2_{BMO}$: a subset of $\bbH^2$ with $Z$ satisfying $Z\cdot W:=\int_0^\cdot Z_s  \!\cdot dW_s $ is a Bounded Mean Oscillation (BMO) martingale, that is, there exists a non-negative constant $C$ such that, for each stopping time $\tau\leq T$,
		\begin{equation*}
			\bbE_{\tau}\left[\int_\tau^T \abs{Z_r}^2 \, dr\right] \leq C^2,
		\end{equation*}
		where $\bbE_{\tau}$ is the conditional expectation given $\scF_\tau$. The $\norm{Z}_{\bbH^2_{BMO}}:=\norm{\int_0^\cdot Z_s\cdot d W_s}_{BMO}$ is defined as the smallest non-negative constant $C$ for which the above inequality is satisfied.
		\item For a closed subset $C$ of $\mathbb{R}^n$ and $a \in \mathbb{R}^n$, the distance between $a$ and the set $C$ is defined by
		\begin{equation*}
			\dis(a, C) = \min_{b \in C} \abs{a - b}.
		\end{equation*}
		The set $\Pi_C(a)$ consists of all points $b \in C$ at which this minimum is attained, and we refer to this as the projection set.
		\item For a Malliavin differentiable random variable $\xi$, we denote $D$ as the Malliavin derivative and set
		\[
		A_\xi:=\mathop{\mathrm{ess\,sup}}_{(t,\omega)}|D_t\xi(\omega)|.
		\]
	\end{itemize}
	When a different underlying measure other than $\bbP$ is involved, we indicate it explicitly in
	the notation, for example by writing $\E^{\bbQ}[\cdot]$ and
	$\|\cdot\|_{BMO(\bbQ)}$ for expectations and BMO norms taken with respect to a
	probability measure $\bbQ$. Here are some inequalities from BMO martingales theory; all the proofs can be found in \cite{kazamaki2006BMO}.
	\begin{lemma}[Reverse H\"older Inequality (RHI)]\label{lem:RHI-BMO}
		Let $M$ be a one–dimensional continuous $\mathrm{BMO}$ martingale under $\bbP$.
		Then the Dol\'eans–Dade exponential $\mathcal{E}(M)$ satisfies the reverse
		H\"older inequality: there exists $q^*>1$ such that for all $q\in(1,q^*)$ and
		all stopping times $\tau\le T$,
		\begin{equation}\label{eq:RH-explicit}
			\bbE\!\left[\left.
			\left(\frac{\mathcal{E}(M)_T}{\mathcal{E}(M)_\tau}\right)^{\!q}
			\right|\F_\tau\right]
			\le C_q^*,
		\end{equation}
		where the constants $q^*$ and $C_q^*$ depend explicitly on the $\mathrm{BMO}$ norm
		$\|M\|_{\mathrm{BMO}(\prob)}$:
		\[
		q^* := \phi^{-1}\!\Big(\|M\|_{\mathrm{BMO}}\Big),
		\qquad
		\phi(x) := \Big(1+\tfrac{1}{x^2}\log\frac{2x-1}{2x-2}\Big)^{1/2}-1,
		\]
		and for $1<q<q^*$,
		\[
		C_q^* :=
		2\Bigg(1 - \frac{2q-2}{2q-1}\,
		\exp\!\Big\{ q^2\big(\|M\|_{\mathrm{BMO}}^2 + 2\|M\|_{\mathrm{BMO}}\big) \Big\}\Bigg)^{-1}.
		\]
		In particular, $q^*$ decreases and $C_q^*$ increases with respect to
		$\|M\|_{\mathrm{BMO}}$.
	\end{lemma}
	\begin{lemma}\label{lem:BMO-change-measure}
		For $K>0$, there exist constants $c_1(K),c_2(K)>0$, depending only on $K$,
		such that the following holds. Let $N$ be a one-dimensional continuous
		$\mathrm{BMO}$ martingale under $\bbP$ with
		$\|N\|_{\mathrm{BMO}(\prob)}\le K$, and define the equivalent measure
		$\tilde{\bbP}$ by
		\[
		\frac{d\tilde{\bbP}}{d\bbP} = \mathcal{E}(N)_T.
		\]
		For any continuous $\mathrm{BMO}$ martingale $M$ under $\bbP$, set $\tilde M := M - \langle M,N\rangle$.
		Then
		$\tilde M\in \mathrm{BMO}(\tilde{\bbP})$ and
		\begin{equation}\label{eq:BMO-comparability}
			c_1(K)\,\|M\|_{\mathrm{BMO}(\bbP)}
			\le
			\|\tilde M\|_{\mathrm{BMO}(\tilde{\bbP})}
			\le
			c_2(K)\,\|M\|_{\mathrm{BMO}(\bbP)}.
		\end{equation}
	\end{lemma}
	\begin{definition}[sliceability]\label{def:sliceability}
		Let $M$ be a one–dimensional continuous $\mathrm{BMO}$ martingale under $\bbP$ on $[0,T]$.
		A \emph{random partition} of $[0,T]$ is a finite sequence of stopping times
		$0=\tau_0\le\tau_1\le\cdots\le\tau_m=T$. For $\delta>0$, the \emph{index of sliceability} of $M$ is the function
		\[
		N_M:(0,\infty)\longrightarrow \bbN\cup\{\infty\}
		\]
		defined as follows: $N_M(\delta)$ is the smallest integer $m$ such that there
		exists a random partition $(t_k)_{k=0}^m$ of $[0,T]$ with
		\begin{equation}\label{eq:slice-M}
			\big\|M^{(k)}\big\|_{\mathrm{BMO}} \le \delta,
			\qquad 1\le k\le m,
		\end{equation}
		where the $k$–th increment $M^{(k)}$ is the continuous martingale
		\[
		M^{(k)}_t := M_{t\wedge \tau_k} - M_{t\wedge\tau_{k-1}},
		\qquad 0\le t\le T.
		\]
		If no such $m$ exists, we set $N_M(\delta)=\infty$. We say that $M$ is \emph{$\delta$–sliceable} if $N_M(\delta)<\infty$, and simply
		\emph{sliceable} if $M$ is $\delta$–sliceable for every $\delta>0$.
	\end{definition}
	\subsection{Statement of the main theorem}
	In this paper, we consider the following general mean-field BSDE
	\begin{equation} \label{mainBSDE}
		Y_t=\xi+\int_t^T \int_{\tilde{\Omega}} g(s,Z_s, \tilde{Z}_s) d \tilde{\bbP} \, ds-\int_t^T Z_s  \!\cdot dW_s, \quad t \in [0, T].
	\end{equation}
	The main result of this paper establishes the existence and uniqueness of solutions to the mean-field BSDE under a bounded terminal condition. To this end, we introduce the following assumptions. Let \( K_0, K_z, K_{\tilde{z}} > 0 \) be fixed constants.
	
	\begin{enumerate}
		\item[{$\Hzero$}] The random variable $\xi$ is bounded almost surely by $K_0 > 0$, and likewise, the deterministic supremum $\sup_{t\in[0,T]}|g(t,0,0)|$ is also bounded by the same constant $K_0 > 0$.
		\item[$\Hqunif$] There exist constants $K_z > 0$ and $K_{\tilde{z}} > 0$ such that:
		\begin{align*}
			|g(t,z,\tilde z)-g(t,z',\tilde z')|
			\le & K_z\big(1+|z|+|z'|+|\tilde z|+|\tilde z'|\big)\big(|z-z'|\big) \notag \\ &+ K_{\tilde z}\big(1+|z|+|z'|+|\tilde z|+|\tilde z'|\big)\big(|\tilde z-\tilde z'|\big),
		\end{align*} for any $(t, z, \tilde{z}, z', \tilde{z}')\in [0,T] \times [\bbR^d \times \bbR^d]^2$.
		\item[$\Hqsep$] There exist constants $K_z > 0$ and $K_{\tilde{z}} > 0$ such that:
		\begin{align*}
			\abs{g(t, z, \tilde{z}) - g(t, z', \tilde{z}')} \leq & K_z(1+\abs{z} + \abs{z'}+\abs{\tilde{z}} + \abs{\tilde{z}'})\abs{z-z'} + K_{\tilde{z}}(1+\abs{\tilde{z}} + \abs{\tilde{z}'})\abs{\tilde{z}-\tilde{z}'},
		\end{align*} for any $(t, z, \tilde{z}, z', \tilde{z}')\in [0,T] \times [\bbR^d \times \bbR^d]^2$.

        \item[$\Hbound$] There are constants $A_0,A_1\geq0$ such that
        \[
        |g(t,z,\tilde z)-g(t,0,\tilde z)|\leq A_0+A_1|z|^2
        \quad\text{for all }(t,z,\tilde z).
        \]
	\end{enumerate}
	\begin{remark}
		The unified quadratic condition ${\Hqunif}$ is the fully coupled case, in contrast with the separately quadratic ${\Hqsep}$. It is immediate that ${\Hqsep}$ implies ${\Hqunif}$. Any statement proved under ${\Hqunif}$ also holds under ${\Hqsep}$, while the converse implication between the assumptions is not true.  We provide an explicit example here. For simplicity, consider the one-dimensional case and define
		\[
		g(t,z,\tilde z) := z\,\tilde z, \qquad (t,z,\tilde z)\in[0,T]\times\bbR^2.
		\]
		This $g$ satisfies ${\Hqunif}$, but does not satisfy ${\Hqsep}$.
	\end{remark}
	\begin{remark}
		Note that ${\Hzero}$ and ${\Hqunif}$ implies ${\Hbound}$. On the other hand, ${\Hzero}$ and ${\Hqsep}$ does not necessarily imply ${\Hbound}$.
	\end{remark}
	\begin{example}
		The following $g$ satisfies the condition ${\Hqsep}$ and ${\Hbound}$:
		\[
		g(t,z,\tilde z)=\min\{|z|^2,|\tilde z|^2\}+\min(|z|(1+|\tilde z|),1)+l(t, z,\tilde z)+q_1(t,z)+q_2(t,\tilde z),
		\]
		where $l$ is a Lipschitz function, $q_1,q_2$ are quadratic polynomials.
The nonadditive interaction should be distinguished from additive
quadratic mean-field terms treated in \cite{hao2022mean}.
	\end{example}
	Let us state our main theorems. The first main theorem assumes a Markovian structure with ${\Hqunif}$ while the second one assumes a non-Markovian structure with ${\Hqsep}$ and $\Hbound$.
	\begin{theorem}\label{thm:mainHolder}
		Assume that there is a constant $C$ such that $b:[0,T]\times\bbR^k\to\bbR^k$ and $\sigma:[0,T]\to \bbR^{k\times d}$ are measurable functions with 
		\begin{itemize}
			\item 	$|b(t,x)-b(t,x')|\leq C|x-x'|$, $|b(t,0)|\leq C$ for every $(t,x)\in[0,T]\times\bbR^k$, and
			\item $C^{-1}|v|^2\leq v^\intercal(\sigma\sigma^\intercal)(t)v\leq C|v|^2$ for all $v\in\bbR^k$ and $t\in[0,T]$.
		\end{itemize} Then, let $X$ be the strong solution of 
		\[
		dX_t=b(t, X_t)dt+\sigma(t)dW_t;\qquad X_0=x_0.
		\]
		Let us assume ${\Hzero}$ and ${\Hqunif}$. In addition, assume that
		$$\xi=\phi(X_{t_1},X_{t_2},\cdots, X_{t_n}),$$ 
		at deterministic observation times $0\leq t_1<\cdots<t_n\leq T$, for a bounded function $\phi:(\bbR^k)^n\to\bbR$ which is $\alpha$-H\"older continuous, that is, there exists $K>0$ and $\alpha\in(0,1]$ such that 
		\[
		|\phi(x)-\phi(x')|\leq K\max_{i=1,2,\cdots, n}|x_i-x'_i|^\alpha,
		\]
		for any $x,x'\in (\bbR^k)^n$. 
		Then, there exists a unique solution \((Y,Z) \in \bbS^\infty \times \bbH^2_{\mathrm{BMO}}\) to the BSDE \eqref{mainBSDE}.
	\end{theorem}
	\begin{remark}
		Note that our terminal condition is not necessarily Malliavin differentiable if $\alpha<1$.
	\end{remark}
	\begin{theorem}\label{thm:mainPath}
		Let Assumptions ${\Hzero}$, ${\Hqsep}$, and $\Hbound$ hold. In addition, assume that there exists a function $\phi:C([0,T];\bbR^d)\to\bbR$ such that $\xi=\phi(W_{[0,T]})$ and $\phi$ is uniformly continuous in the supremum norm.
		Then, there exists a unique solution \((Y,Z) \in \bbS^\infty \times \bbH^2_{\mathrm{BMO}}\) to the BSDE \eqref{mainBSDE}.
	\end{theorem}
	\begin{remark}
		\label{rem:shifted-generator}
		In the statement and proof of our main theorem, we may, without loss of generality, assume that
		\[
		g(t,0,\tilde z) = 0, \qquad \forall\,(t,\tilde z)\in[0,T]\times\bbR^d.
		\]
		Given any generator $g$ satisfying Assumptions $\Hzero$–${\Hqunif}$, we may  introduce the shifted generator
		\[
		\hat g(t,z,\tilde z)
		:= g(t,z,\tilde z) - g(t,0,\tilde z),
		\]
		and consider the mean-field BSDE driven by $\hat g$. For any solution $(Y,Z)$ of the original BSDE with generator $g$, the process
		\[
		U_t := Y_t - \int_t^T \widetilde{\E}\big[g(s,0,\tilde Z_s)\big]\,ds, \qquad 0\le t\le T,
		\]
		together with the same $Z$, solves the BSDE with generator $\hat g$. Conversely, given a solution $(U,Z)$ to the BSDE with generator $\hat g$, the process
		\[
		Y_t = U_t + \int_t^T \widetilde{\E}\big[g(s,0,\tilde Z_s)\big]\,ds
		\]
		solves the original BSDE. The shift is deterministic and bounded because
$Z\in\bbH^2$ and $g(t,0,\tilde z)$ has quadratic growth. Thus there is a
one-to-one correspondence in terms of existence and uniqueness.
The shifted generator retains $\Hqunif$ or $\Hqsep$ after enlarging the
constants, for example replacing $K_{\tilde z}$ by $2K_{\tilde z}$;
subsequent constants refer to the shifted generator.
The estimates for $Y$ in Sections \ref{Sec: MalliavinD}--\ref{Sec: ExistenceUniqueness}
are estimates for the normalized process $U$. To return to the original
variables one adds back the displayed shift. For two solutions its
difference is bounded by
\[
 K_{\tilde z}\bigl(\sqrt T+\|Z^1\|_{\bbH^2}
                         +\|Z^2\|_{\bbH^2}\bigr)
                   \|Z^1-Z^2\|_{\bbH^2}.
\]
Thus the stability and convergence conclusions below also hold for the
original $Y$, with the corresponding adjusted constants. \emph{From this point on, we therefore work directly with the shifted generator, still denoted by $g$, and silently assume
			$g(t,0,\tilde z)=0$ for all $(t,\tilde z)\in[0,T]\times\bbR^d$.}
	\end{remark}
	\section{Terminal conditions with bounded Malliavin derivative}\label{Sec: MalliavinD}
	
	This section is devoted to deriving tractable upper bounds for the solution \((Y, Z)\) of the BSDE \eqref{mainBSDE} under the assumption that the driver \(g\) is Lipschitz continuous and the terminal condition \(\xi\) is Malliavin differentiable. These bounds under the stronger condition become particularly useful in the proof of Lemma \ref{Lem: Malliavin}, we will show that the quadratic-generator case with the terminal conditions considered below, can be reduced to this setting through suitable approximation techniques. Therefore, we initially work under stronger regularity assumptions to facilitate this reduction. For this purpose, we suppose all along this section that $g$ satisfies the following classical Lipschitz condition, and a condition on  \(\xi\) will be introduced in Proposition \ref{prop: lipbound} below.
	
	\begin{itemize}
		\item[$\Hlip$] There exist a constant $L_z > 0$ such that:
		\begin{equation*}
			\abs{g(t, z, \tilde{z}) - g(t, z', \tilde{z}')} \leq L_z \brak{ \abs{z- z'}+\abs{\tilde{z}-\tilde{z}'}},
		\end{equation*}
		for any $(t, z, \tilde{z}, z', \tilde{z}')\in [0,T] \times [\bbR^d \times \bbR^d]^2$.
	\end{itemize}
	
	The existence and uniqueness of a solution \((Y, Z) \in \bbS^2 \times \bbH^2\) to the BSDE \eqref{mainBSDE} under assumptions $\Hzero$–$\Hlip$ is a well-established result, as shown in Theorem 3.1 of \cite{Buckdahn2009MFBSDE}. In the following proposition, we demonstrate that the process \(Y\), which lies in \(\bbS^\infty\), is bounded by a constant that is independent of the Lipschitz constant \(L_z\). More remarkably, we show that the same independence holds for the process \(Z\), provided the terminal condition \(\xi\) admits uniformly bounded Malliavin derivatives.
	
	\begin{proposition}\label{prop: lipbound}
		If $\Hzero$ and $\Hlip$ hold, then the unique solution $(Y, Z) \in \bbS^2 \times \bbH^2$ of the Lipschitz mean-field BSDE satisfies the following properties.
		\begin{enumerate}
			\item $Y \in \bbS^\infty$ with the upper bound:
			\begin{equation}\label{eq: lipboundy}
				\|Y\|_{\bbS^\infty}
				\le \norm{\xi}_{L^\infty}.
			\end{equation}
			\item If $\xi\in\mathbb D^{1,2}$ and $A_\xi<\infty$,
then $Z\in\mathbb H^\infty$, with a bound independent of $L_z$:
\begin{equation}\label{eq: lipboundz}
 \|Z\|_{\mathbb H^\infty}\leq A_\xi.
\end{equation}
		\end{enumerate}
	\end{proposition}
	
	\begin{proof}
Throughout the proof, the generator is normalized as in
Remark \ref{rem:shifted-generator}. Fix the solution $(Y,Z)$ and define
\[
 F(t,z):=\widetilde{\mathbb E}\big[g(t,z,\widetilde Z_t)\big].
\]
This is a deterministic measurable function of $(t,z)$, is globally
$L_z$-Lipschitz in $z$, and satisfies $F(t,0)=0$. Thus $(Y,Z)$ also solves
the ordinary Lipschitz BSDE with generator $F$.
Define the predictable process
\[
 b_t:=\begin{cases}
 F(t,Z_t)Z_t/|Z_t|^2,&Z_t\neq0,\\
 0,&Z_t=0.
 \end{cases}
\]
Then $|b_t|\leq L_z$. Under the probability measure with density
$\mathcal E(\int_0^\cdot b_s\cdot dW_s)_T$, the process
$W_t-\int_0^t b_s\,ds$ is Brownian. The representation formula for linear
Lipschitz BSDEs gives $Y_t=\mathbb E_t^{\mathbb Q}[\xi]$, proving
\eqref{eq: lipboundy}.

Suppose now that $\xi\in\mathbb D^{1,2}$ and $A_\xi<\infty$.
First assume that $F(t,\cdot)$ is continuously differentiable.
Malliavin differentiation of this ordinary BSDE yields, for almost every
$\theta$ and $\theta\leq t\leq T$,
\[
 D_\theta Y_t=D_\theta\xi+
 \int_t^T \partial_zF(s,Z_s)\cdot D_\theta Z_s\,ds
 -\int_t^T D_\theta Z_s\cdot dW_s.
\]
The identity is interpreted componentwise for the vector $D_\theta Y_t$.
Since $|\partial_zF|\leq L_z$, the corresponding linear BSDE representation
under the measure with density
$\mathcal E(\int_0^\cdot\partial_zF(s,Z_s)\cdot dW_s)_T$ gives
\[
 D_\theta Y_t=\mathbb E_t^{\mathbb Q}[D_\theta\xi],
 \qquad |D_\theta Y_t|\leq A_\xi.
\]
Proposition 5.3 of \cite{el1997backward} tells us that  $Z_t=D_tY_t$, valid
$dt\otimes d\mathbb P$-almost everywhere for Lipschitz BSDEs, implies
\eqref{eq: lipboundz}.

For a merely Lipschitz $F$, let $\rho_\varepsilon$ be a smooth spatial
mollifier and set
\[
 F_\varepsilon(t,z)
 =(F(t,\cdot)*\rho_\varepsilon)(z)
 -(F(t,\cdot)*\rho_\varepsilon)(0).
\]
These generators satisfy $F_\varepsilon(t,0)=0$, have Lipschitz constant
at most $L_z$, and converge uniformly to $F$ as $\varepsilon\downarrow0$.
For the ordinary BSDEs with terminal value $\xi$ and generators
$F_\varepsilon$, the preceding argument gives
$\|Z^\varepsilon\|_{\mathbb H^\infty}\leq A_\xi$.
Lipschitz BSDE stability yields convergence to $(Y,Z)$ in
$\mathbb S^2\times\mathbb H^2$. Passing to a
$dt\otimes d\mathbb P$-almost everywhere convergent subsequence preserves
the bound and proves \eqref{eq: lipboundz} in full generality.
\end{proof}
	\section{A priori estimates and stability for BSDEs}\label{Sec: Stability}
	In this section, we establish stability estimates for the mean-field BSDE \eqref{mainBSDE} under two different quadratic structures of the generator. More precisely, we prove Theorem~\ref{thm: stability} under the \emph{unified} quadratic condition $\Hqunif$ and H\"older cylindrical terminal conditions, and Theorem \ref{thm: stability2} under the \emph{separable} quadratic condition $\Hqsep$. Both results yield the same stability estimate in the Banach space \(\bbS^\infty\times\bbH^2_{\mathrm{BMO}}\). 

	\begin{Remark}
		Under assumptions $\Hzero$ and ${\Hqunif}$, we have the following estimate:
		\begin{align*}
			|g(t,z,\tilde z)|
			&\le K_0 + K_z|z| + K_z|z|^2 + (K_z+K_{\tilde z})|z||\tilde z|
			+ K_{\tilde z}|\tilde z| + K_{\tilde z}|\tilde z|^2,
		\end{align*}
		for any $(t,z,\tilde z)\in[0,T]\times\bbR^d\times\bbR^d$, where $K_0$ is the bound from $\Hzero$. Taking expectations with respect to $\tilde{\bbP}$ yields:
		\begin{equation}\label{estimate}
			\tilde{\bbE}\,|g(t,Z_t,\tilde Z_t)|
			\le K_0 + \frac12 K_z + \frac12 K_{\tilde z}
			+ \Big(2K_z+\frac12K_{\tilde z}\Big)|Z_t|^2
			+ \Big(2K_{\tilde z} + \frac12K_z\Big)\tilde{\bbE}|\tilde Z_t|^2,
		\end{equation}
		where we used symmetry $\tilde{\bbE}|\tilde Z_t|=\E|Z_t|$, together with $|Z_t|\,\E|Z_t|
		\le \frac12|Z_t|^2+\frac12\E|Z_t|^2$.
	\end{Remark}
	
	We begin by noting that if the solution \( Y \in \bbS^\infty \), then for any solution pair \( (Y, Z) \in \bbS^\infty \times \bbH^2 \) to the quadratic mean-field BSDE, the stochastic integral \( \int_0^\cdot Z_s  \!\cdot dW_s \) defines a BMO martingale. A bound is available in terms of the uniform norm of $Y$ and the individual $\bbH^2$ norm of $Z$. A uniform bound over families requires an additional estimate, supplied below by $\Hbound$ or by a deterministic gradient envelope. 
	
	\begin{proposition}\label{Prop: estimateZ}
Suppose that $\Hzero$ and $\Hqunif$ hold, and let
$(Y,Z)\in\bbS^\infty\times\bbH^2$ solve \eqref{mainBSDE}. Then
$Z\cdot W$ is a BMO martingale. More precisely, put
\[
 A=2K_z+\tfrac12K_{\tilde z},\quad
 B=2K_{\tilde z}+\tfrac12K_z,\quad
 K=K_0+\tfrac12(K_z+K_{\tilde z}),\quad
 M=\|Y\|_{\bbS^\infty},\quad \lambda=2A,
\]
and define
\[
 \Phi_\lambda(x)=\frac{e^{\lambda|x|}-1-\lambda|x|}{\lambda^2}.
\]
Then
\begin{equation}\label{eq:BMOZ-Hqprime}
 \|Z\cdot W\|_{\mathrm{BMO}}^2
 \le 2\Phi_\lambda(M)
 +2\frac{e^{\lambda M}-1}{\lambda}
       \left(KT+B\|Z\|_{\bbH^2}^2\right).
\end{equation}
In addition, if we assume $\Hbound$,
then, with $\theta=2A_1+1$,
\begin{equation}\label{eq:uniform-bmo-HB}
 \|Z\cdot W\|_{\mathrm{BMO}}^2
 \le 2\Phi_\theta(M)
       +2A_0T\frac{e^{\theta M}-1}{\theta},
\end{equation}
where $M=\|U\|_{\bbS^\infty}$.
\end{proposition}
\begin{proof}
In this proof we assume $f(t,0,\tilde z)=0$ via Remark \ref{rem:shifted-generator}. The function $\Phi_\lambda$ is $C^2$ on $\bbR$, is nonnegative, and satisfies
\[
 \Phi_\lambda''(x)-\lambda|\Phi_\lambda'(x)|=1,
 \qquad
 |\Phi_\lambda'(x)|\le\frac{e^{\lambda M}-1}{\lambda}
 \quad (|x|\le M).
\]
Write $f_s=\widetilde{\bbE}g(s,Z_s,\widetilde Z_s)$.
For any stopping time $\tau\le T$, It\^o's formula gives
\[
 \Phi_\lambda(Y_\tau)
 =\Phi_\lambda(\xi)
 +\int_\tau^T\left(\Phi_\lambda'(Y_s)f_s
               -\tfrac12\Phi_\lambda''(Y_s)|Z_s|^2\right)ds
 -\int_\tau^T\Phi_\lambda'(Y_s)Z_s\cdot dW_s.
\]
The stochastic integral is square integrable because $Y$ is bounded and
$Z\in\bbH^2$. By \eqref{estimate},
$|f_s|\le K+A|Z_s|^2+B\bbE|Z_s|^2$. Taking the conditional expectation,
using $\lambda=2A$ and $\Phi_\lambda\ge0$, yields
\[
 \tfrac12\bbE_\tau\int_\tau^T|Z_s|^2ds
 \le\Phi_\lambda(M)+\frac{e^{\lambda M}-1}{\lambda}
       \left(KT+B\int_0^T\bbE|Z_s|^2ds\right).
\]
The supremum over $\tau$ proves \eqref{eq:BMOZ-Hqprime}, including BMO
membership. Under $\Hbound$, use
$|f_s|\le A_0+A_1|Z_s|^2$ and
\[
 \Phi_\theta''-2A_1|\Phi_\theta'|
 =1+(\theta-2A_1)|\Phi_\theta'|\ge1
\]
in the same calculation. This proves \eqref{eq:uniform-bmo-HB}.
\end{proof}
	
	Next, suppose that the solution satisfies \( Z \in \bbH^2_{\mathrm{BMO}} \). Then, for any solution pair \( (Y, Z) \in \bbS^2 \times \bbH^2_{\mathrm{BMO}} \) to the quadratic mean-field BSDE, we derive the following proposition showing that \( Y \in \bbS^\infty \), with a bound that is independent of the BMO norm \( \norm{\int_0^\cdot Z_s  \!\cdot dW_s }_{\mathrm{BMO}} \). 
	
	
	\begin{proposition} \label{Prop: estimateY}
		Suppose Assumptions $\Hzero$ and ${\Hqunif}$ hold. Then for any solution
		\( (Y, Z) \in \bbS^2 \times \bbH^2_{\mathrm{BMO}} \) of BSDE \eqref{mainBSDE}, we have \( Y \in \bbS^\infty \) and
		\begin{equation}
			\norm{Y}_{\bbS^\infty} \leq \norm{\xi}_{L^\infty}.
		\end{equation}
	\end{proposition}
	\begin{proof}
		The proof relies on a linearization argument combined with the BMO property.
		Using Remark \ref{rem:shifted-generator}, without loss of generality we assume
		\[
		g(t,0,\tilde z)=0,\qquad \forall (t,\tilde z)\in[0,T]\times\bbR^d.
		\]
		Then BSDE \eqref{mainBSDE} can be rewritten as
		\begin{equation}\label{linearBSDE}
			Y_t
			= \xi + \int_t^T \Big( \int_{\tilde{\Omega}} L_s \, d\tilde{\bbP} \Big) \cdot Z_s \, ds
			- \int_t^T Z_s  \!\cdot dW_s,
			\qquad t \in [0, T],
		\end{equation}
		where the process $(L_s)_{0\le s\le T}$ is defined by
		\[
		L_s:=\frac{g(s, Z_s, \tilde Z_s) - g(s, 0, \tilde Z_s)}{|Z_s|^2}\, Z_s\, \mathbf{1}_{\{Z_s\neq 0\}},
		\qquad 0 \leq s \leq T.
		\]
		Since $g(s,0,\tilde z)=0$ by construction, the numerator is simply $g(s,Z_s,\tilde Z_s)$. Moreover, by Assumption ${\Hqunif}$, we have for any $(s,z,\tilde z)$,
		\[
		|g(s,z,\tilde z)-g(s,0,\tilde z)|
		\le K_z\big(1+|z|+2|\tilde z|\big)\,|z|.
		\]
		Therefore,
		\begin{equation}\label{eq:bound-EL-Hqprime}
			\Big|\int_{\tilde{\Omega}} L_s \, d\tilde{\bbP}\Big|
			\le K_z\big(1+|Z_s|+2\E|Z_s|\big),\qquad 0\le s\le T.
		\end{equation}
		Define the process
		\[
		M_t:=\int_0^t \Big( \int_{\tilde{\Omega}} L_s \, d\tilde{\bbP} \Big)\cdot dW_s,\qquad t\in[0,T].
		\]
		We claim that $M$ is a BMO martingale. Indeed, by \eqref{eq:bound-EL-Hqprime} and $(a+b+c)^2\le 3(a^2+b^2+c^2)$,
		\[
		\Big|\int_{\tilde{\Omega}} L_s \, d\tilde{\bbP}\Big|^2
		\le 3K_z^2\Big(1+|Z_s|^2+4\E|Z_s|^2\Big).
		\]
		Hence, for any stopping time $\tau$,
		\begin{align*}
			\E\!\Big[\langle M\rangle_T-\langle M\rangle_\tau\,\big|\,\F_\tau\Big]
			&=\E\!\Big[\int_\tau^T \Big|\int_{\tilde{\Omega}} L_s \, d\tilde{\bbP}\Big|^2 ds \,\Big|\,\F_\tau\Big] \\
			&\le 3K_z^2\Big(
			T + \E\!\Big[\int_\tau^T |Z_s|^2ds\,\Big|\,\F_\tau\Big] + 4\int_0^T \E|Z_s|^2ds
			\Big).
		\end{align*}
		Since $Z\in\bbH^2_{\mathrm{BMO}}\subset\bbH^2$, $\|M\|_{\mathrm{BMO}}<\infty$, proving the claim.
		By Kazamaki's condition, the stochastic exponential $\mathcal{E}(M)$ is a uniformly integrable martingale.
		Define the probability measure $\bbP^L$ by $d\bbP^L:=\mathcal{E}(M)_T\,d\bbP$; then $\bbP^L$ is equivalent to $\bbP$.
		By Girsanov's theorem, the process
		\[
		W_t^L := W_t - \int_0^t \Big(\int_{\tilde{\Omega}} L_s \, d\tilde{\bbP} \Big)\, ds,\qquad t\in[0,T],
		\]
		is a $d$--dimensional Brownian motion under $\bbP^L$, and \eqref{linearBSDE} becomes
		\[
		Y_t=\xi-\int_t^T Z_s\cdot dW_s^L,\qquad 0\le t\le T.
		\]
        Lemma \ref{lem:BMO-change-measure} shows that $Z\cdot W^L$
        is BMO under $\bbP^L$, so the stochastic integral is a true
        martingale, thus
		\[
		Y_t=\E^{\bbP^L}\big[\xi\,\big|\,\F_t\big],\qquad 0\le t\le T,
		\]
		and consequently $|Y_t|\le \|\xi\|_{L^\infty}$ for all $t\in[0,T]$, $\bbP^L$--a.s.
		Since $\bbP^L\sim \bbP$, the same bound holds $\bbP$--a.s.,
		and taking the essential supremum over $t\in[0,T]$ completes the proof.
	\end{proof}
	We now state a direct consequence of the previous two propositions.
\begin{corollary}\label{Cor: YZ estimates}
Suppose that $\Hzero$ and $\Hqunif$ hold. If a solution of the normalized
equation belongs either to $\bbS^\infty\times\bbH^2$ or to
$\bbS^2\times\bbH^2_{\mathrm{BMO}}$, then it belongs to
$\bbS^\infty\times\bbH^2_{\mathrm{BMO}}$ and
$\|Y\|_{\bbS^\infty}\le\|\xi\|_{L^\infty}$.
If $\Hbound$ also holds, then
\begin{equation}\label{eq:YZ-joint-bound}
 \|Y\|_{\bbS^\infty}+\|Z\cdot W\|_{\mathrm{BMO}}
 \le K_0+\left[2\Phi_\theta(K_0)
       +2A_0T\frac{e^{\theta K_0}-1}{\theta}\right]^{1/2},
 \qquad\theta=2A_1+1.
\end{equation}
\end{corollary}
	\begin{proposition}\label{prop:qvbound}
	Assume $(Y,Z)\in\bbS^\infty\times\bbH^2_{\mathrm{BMO}}$ is a solution to the BSDE \eqref{mainBSDE} with $\xi$ satisfying:
	\begin{itemize}
		\item[$({\rm H_{\xi}})$] \[\xi=\phi(W_{t_1},\cdots, W_{t_n})\]
		for a bounded function $\phi:(\bbR^d)^n\to\bbR$ which is $\alpha$-H\"older continuous, that is, there exists $K>0$ and $\alpha\in(0,1]$ such that 
		\[
		|\phi(x)-\phi(x')|\leq K\max_{i=1,2,\cdots, n}|x_i-x'_i|^\alpha,
		\]
		for any $x,x'\in (\bbR^d)^n$.
	\end{itemize}
	Then $Z\cdot W$ is sliceable via deterministic time and moreover, there exists a constant $C_\alpha$ such that $$\norm{\int_0^T|Z_s|^2ds}_{L^\infty}\leq C_\alpha T^\alpha.$$
\end{proposition}
\begin{proof}
	First, let us assume $\phi(x)$ only depends on $x_n$ and define a function $h:\bbR^d\to\bbR$ as $h(x)=\phi(0,0,\cdots,0,x)$.
	Throughout this proof, we will denote a deterministic function
	\[
	f(s,z):=\int_{\tilde\Omega}g(s, z, \tilde Z_s)d\tilde\bbP.
	\]
	where $\tilde Z_s$ is an independent copy of $Z_s$ on the probability space $(\tilde \Omega,\tilde\scF,\tilde\bbP)$.
	
	Let us define a sequence of mollifiers. Let $\eta\in C^\infty_c(\bbR^d;\bbR)$ be a probability density with support  $\crl{x\in\bbR^d:|x|\leq \veps}$ and
	\[
	\eta_\veps(x):=\veps^{-d}\eta(x/\veps)
	\]
	converges to the Dirac delta function as $\veps\searrow0$.
	
	We will use the coupling argument from \cite{song2021gradient}.
	First, let us mollify the terminal condition and the driver so that
	\[
	[h^\veps]_\alpha\leq [h]_\alpha,\quad\text{ and } \abs{\frac{dh^\veps}{dx}}\leq L^\veps
	\]
	where $h^\veps:=h\ast\eta_\veps, f^\delta:=f\ast\eta_\delta$ are spatial mollifications and $L^\veps>0$ is a constant. The BSDE
	\begin{align}\label{prop:qvbound_ pf: original}
		Y^{\veps,\delta}_t=h^\veps(W_T)+\int_t^T f^\delta(s, Z^{\veps,\delta}_s)ds-\int_t^TZ^{\veps,\delta}_s\cdot dW_s,
	\end{align}
	has a unique strong solution such that $Y^\veps$ is bounded and
	the nonlinear Feynman-Kac theorem for quadratic BSDE tells us that
	there exists a function $u^\veps\in C^{1,2}([0,T]\times\bbR^d;\bbR)$ such that 
	\begin{align*}
		Y_t^{\veps,\delta}= u^{\veps,\delta} (t, W_t)\quad\text{and}\quad Z_t^{\veps,\delta}= \nabla u^{\veps,\delta}(t, W_t)
	\end{align*}
	and $u^{\veps,\delta}$ is the solution of
	\begin{align}\label{eq:pde}
	\partial_tu^{\veps,\delta}+\half\Delta u^{\veps,\delta}+f^\veps(s,\nabla u^{\veps,\delta})=0;\qquad u^{\veps,\delta}(T,x)=h^{\veps,\delta}(x)	
	\end{align}
		with $|u^{\veps,\delta}|, |\nabla u^{\veps,\delta}|$ being bounded: see \cite{cheridito2014bsdes} and \cite{pardoux1992backward}. Note that $\nabla u^{\veps,\delta}$ is locally Lipschitz uniformly in $t\in[0,T]$ because $u^{\veps,\delta}$ is twice continuously differentiable.
	
	Let us construct a system of forward SDEs:
	\begin{equation}
		\begin{aligned}\label{prop:qvbound_pf:coupled}
			dX_t&=v(t,X_t,\hat X_t)dt+dW_t, &X_r&= x\\
			d\hat X_t&=v(t,X_t, \hat X_t)dt+ (I-2ee^\intercal)dW_t, &\hat X_r&=\hat x
		\end{aligned}
	\end{equation}
	for $t\in[r,T]$, 
	where $e=(\hat x- x)/|\hat x - x|$ and $v$ is defined by the mean value theorem:
	\begin{align*}
		f^\delta(t,\nabla u^{\veps,\delta}(t,\hat x))&-f^\delta(t,\nabla u^{\veps,\delta}(t,x)) = v(t,\hat x, x)(\hat x- x);\quad\text{ or, in other words}\\
		v(t,\hat x, x)&:=\int_0^1\nabla_z f^\delta(s,\lambda\nabla u^{\veps,\delta}(s,\hat x)+(1-\lambda)\nabla u^{\veps,\delta}(s,x))d\lambda.
	\end{align*}
	Since $\nabla u^{\veps,\delta}(t,x)$ is locally Lipschitz in $x$ uniformly in $t\in[0,T]$, $v(t,\hat x, x)$ is locally Lipschitz in $\hat x$ and $x$ uniformly in $t\in[0,T]$. Therefore, Theorem V.2.5 in \cite{karatzas1998brownian} establishes the pathwise uniqueness of the solution to \eqref{prop:qvbound_pf:coupled}. On the other hand, Stroock-Varadhan \cite[Proposition 1.14]{cherny2005singular} establishes weak existence. Then, the Yamada-Watanabe theorem guarantees the unique existence of a strong solution for \eqref{prop:qvbound_pf:coupled}.
	
	Applying the Ito formula, we obtain 
	\begin{align*}
		u^{\veps,\delta}(r, \hat x)-u^{\veps,\delta}(r, x)=h^\veps(\hat X_T)-h^\veps(X_T)+\int_r^T\nabla u^{\veps,\delta}(s,\hat X_s)(I-2ee^\intercal)dW_s-\int_r^T\nabla u^{\veps,\delta}(s,X_s)dW_s.
	\end{align*}
	Since $|\nabla u^{\veps,\delta}|$ is bounded, the last two terms are martingales.
	For one-dimensional Brownian motions $B:=e^\intercal W$, $R_s:=(T-r)^{-1/2}B_{(T-r)s}$, and stopping times
	\begin{align*}
		\tau&:=\inf\crl{t>r: \hat X_t=X_t}\\
		\theta&:=\inf\crl{t>0:B_t=\frac{1}{2}|\hat x- x|}\\
		\tilde \theta&:=\inf\crl{t>0:R_t=\frac{1}{2(T-r)^{1/2}}|\hat x- x|},
	\end{align*}
	we have
	\begin{equation}\label{eq:holdercalc}
	\begin{aligned}
	u^{\veps,\delta}(r,\hat x)-u^{\veps,\delta}(r,x)&=\bbE\edg{\brak{h^\veps(\hat X_T)-h^\veps(X_T)}\indicator{\tau\geq T}}\\
	&\leq [h^\veps]_\alpha\bbE\edg{\abs{\hat X_T-X_T}^\alpha \indicator{\tau\geq T}}\\
	&\leq [h]_\alpha \bbE\edg{\abs{(\hat x- x)-2ee^\intercal (W_T-W_r)}^\alpha \indicator{\tau\geq T}}\\
	&\leq 2^\alpha[h]_\alpha \bbE\edg{\abs{\frac{1}{2}|\hat x-x|-B_{T-r}}^\alpha\indicator{\theta\geq T-r}}\\
	&\leq 2^\alpha[h]_\alpha (T-r)^{\frac{\alpha}{2}}\bbE\edg{\abs{\frac{1}{2(T-r)^{1/2}}|\hat x-x|-R_1}^\alpha\indicator{\tilde\theta\geq 1}}.
\end{aligned}
	\end{equation}
	By Lemma \ref{lem:reflection}, there exists a constant $C$ such that
	\[
	\bbE\edg{\abs{\frac{1}{2(T-r)^{1/2}}|\hat x-x|-R_1}^\alpha\indicator{\tilde\theta\geq 1}}\leq C(T-r)^{-1/2}|\hat x-x|,
	\]
	and therefore,
	\begin{align*}
		u^{\veps,\delta}(r,\hat x)-u^{\veps,\delta}(r,x)\leq C(T-r)^{(\alpha-1)/2}|\hat x-x|.
	\end{align*}
	Therefore,
\begin{align}\label{Zboundslice}
		|Z^{\veps,\delta}_t|=|\nabla u^{\veps,\delta} (t, W_t)|\leq C(T-t)^{(\alpha-1)/2}.
\end{align}
	This implies
	\[
	\int_0^T|Z^{\veps,\delta}_s|^2ds\leq C \int_0^T(T-s)^{\alpha-1}ds= \frac{C}{\alpha} T^\alpha.
	\]
	Now, consider the following quadratic BSDEs:
	\begin{align*}
		Y^\delta_t&=h(W_T)+\int_t^Tf^\delta(s, Z_s^\delta)ds-\int_t^T Z_s^\delta\cdot dW_s\\
		Y_t&=h(W_T)+\int_t^Tf(s, Z_s)ds-\int_t^TZ_s\cdot dW_s.
	\end{align*}
	Then, Proposition 2.3 of \cite{briand2013simple} tells us that, for $p>1$,
	\[
	\norm{Y^{\veps,\delta} - Y^\delta}_{\bbS^p}+\norm{Z^{\veps,\delta} -Z^\delta}_{\bbH^p}\leq C_p\norm{h(W_T)-h^\veps(W_T)}_{\bbL^{2p}}
	\]
	where $C_p$ depends only on $p, K_z, K_{\tilde z},$ and $Z$. Since $\int_0^T|Z^{\veps,\delta}_s|^2ds$ is uniformly bounded, if we take $\veps\to0$, we obtain
	\begin{align}\label{prop:qvbound_ pf: Zdelta}
		\int_0^T| Z^\delta_s|^2ds\leq \frac{C}{\alpha}T^\alpha.
	\end{align}
	Now, let us compare the BSDE for $(Y^\delta, Z^\delta)$ with \eqref{prop:qvbound_ pf: original}. Then,
	\begin{align}\label{prop:qvbound_ pf: diff}
		Y^\delta_t - Y_t&=\int_t^T\brak{f^\delta(s, Z_s^\delta)- f(s, Z_s)}ds-\int_t^T\brak{Z^\delta_s-Z_s}\cdot dW_s\\
		&=\int_t^T\brak{f^\delta(s, Z_s^\delta)-f(s, Z_s^\delta)}ds+\int_t^T\brak{f(s, Z_s^\delta)- f(s, Z_s)}ds-\int_t^T\brak{Z^\delta_s-Z_s}\cdot dW_s.
	\end{align}
	Let us define
	\begin{align*}
		U_t&:=Y^\delta_t-Y_t+\int_0^t\brak{f^\delta(s,Z^\delta_s)-f(s,Z^\delta_s)}ds\\
		\xi^\delta&:=\int_0^T\brak{f^\delta(s, Z_s^\delta)-f(s, Z_s^\delta)}ds\\
		dB_t&:=dW_t-\underbrace{\frac{f(t,Z^\delta_t)-f(t,Z_s)}{|Z^\delta_t-Z_t|^2}\brak{Z^\delta_t-Z_t}}_{=:L_t}dt.
	\end{align*}
	Here, $B$ is a Brownian motion under a changed measure $\bbQ$ by the Girsanov theorem since $L\cdot W$ is a BMO martingale. Then, \eqref{prop:qvbound_ pf: diff} becomes
	\begin{align*}
		U_t=\xi^\delta-\int_t^T\brak{Z^\delta_s-Z_s}\cdot dB_s.
	\end{align*}
	Note that, for $\delta<1$,
	\begin{align*}
		\abs{\xi^\delta}&\leq\int_\bbR\int_0^T\abs{f(s, Z_s^\delta-y)-f(s, Z_s^\delta)}ds\eta_\delta(y)dy\\
		&\leq C\int_{-\delta}^\delta\int_0^T(1+|Z^\delta_s|+|y|)|y|ds\eta_\delta(y)dy\\
		&\leq 2C\delta\int_0^T(1+|Z^\delta_s|)ds\leq C'\delta
	\end{align*}
	for some constant $C'$, independent of $\delta$ because of \eqref{prop:qvbound_ pf: Zdelta}.
	Another application of \cite{briand2013simple} gives us
	\begin{align*}
		\norm{U}_{\bbS^p(\bbQ)}+\norm{Z^{\delta} -Z}_{\bbH^p(\bbQ)}\leq C_p\norm{\xi^\delta}_{\bbL^{2p}(\bbQ)}\xrightarrow{\delta\to0}0.
	\end{align*}
	Since $\bbP$ and $\bbQ$ are equivalent and we have \eqref{prop:qvbound_ pf: Zdelta}, we prove that
	\[
	\int_0^T| Z_s|^2ds\leq \frac{C}{\alpha}T^\alpha.
	\]
	The sliceability by deterministic time automatically follows by the same argument with \eqref{Zboundslice}.

	Now let us condsider the general case where $\phi$ depends on $(x_1,\cdots, x_n)$. Without loss of generality, we will assume $\phi$ sufficiently smooth. Note that, there exist a measurable function $u$ such that $Y_{t}=u(t, W_t-W_{t_j}, W_{t_1},\cdots, W_{t_n})$ where $t\in(t_j,t_{j+1})$. Note that $u(t, x, x_1,\cdots, x_n)$  does not depend on $(x_{j+1}, \cdots x_n)$ due to the adaptedness of $Y$. For any constant $\Delta$, it is easy to verify that $P_t=u(t,W_t-W_{t_j}, W_{t_1},\cdots, W_{t_j}+\Delta, W_{t_{j+1}}+\Delta,\cdots, W_{t_n}+\Delta)$ and $Q_t=D_tP_t$ solves the BSDE
	\begin{align*}
		P=\phi(W_{t_1},\cdots, W_{t_j}+\Delta, W_{t_{j+1}}+\Delta,\cdots, W_{t_n}+\Delta)-\int_{t_j}^Tf(s, Q_s)ds-\int_{t_j}^TQ_sdW_s
	\end{align*}
	because the shift $\Delta$ does not affect the increment of $W$ after $t_j$. Therefore, for $\Xi:=\phi(W_{t_1},\cdots, W_{t_j}+\Delta, W_{t_{j+1}}+\Delta,\cdots, W_{t_n}+\Delta)-\phi(W_{t_1},\cdots, W_{t_j}, W_{t_{j+1}},\cdots, W_{t_n})$
	\begin{align*}
		P_t-Y_{t}=\Xi+\int_t^T\brak{f(s, Q_s)-f(s,Z_s)}ds-\int_t^T(Q_s-Z_s)dW_s.
	\end{align*}
	By the standard linearization and Girsanov transform technique, we can find an equivalent measure $\bbQ'$ such that
	\begin{align*}
		\abs{P_t-Y_{t}}=\abs{\bbE^{\bbQ'}_t\edg{\Xi}.}\leq K|\Delta|^\alpha.
	\end{align*}
	Therefore, we can iterate above argument backwards in time to prove the original claim.
\end{proof}
\begin{theorem}\label{thm: stability}
	For $i=1,2$, consider $(Y^i,Z^i)\in\bbS^\infty\times\bbH^2_{\mathrm{BMO}}$ solve
	\[
	Y_t^i=\xi^i+\int_t^T \!\int_{\tilde{\Omega}} g(s,Z_s^i,\tilde{Z}_s^i)\,d\tilde{\bbP}\,ds-\int_t^T Z_s^i\cdot dW_s,\qquad t\in[0,T],
	\]
	with $\xi^1,\xi^2,g$ satisfying $(\mathrm{H}_0), (\mathrm{H}_q)$ and $(\mathrm{H}_\xi)$.
	Then there exists a constant $C> 0$ such that
	\[
	\|Y^1-Y^2\|_{\bbS^{\infty}}+\|Z^1-Z^2\|_{\bbH^2_{\mathrm{BMO}}}
	\ \le\ C\|\xi^1-\xi^2\|_{L^\infty},
	\]
	where $C$ depends only on $(T, K_z, K_{\tilde z}, K_0)$ and $C_\alpha$ in Proposition \ref{prop:qvbound}.
\end{theorem}
\begin{proof}
	Let $\Delta Y:=Y^1-Y^2$, $\Delta Z:=Z^1-Z^2$, $\Delta\xi:=\xi^1-\xi^2$ and
	\[
	\Delta g(s,Z_s,\tilde Z_s):=g(s,Z_s^1,\tilde Z_s^1)-g(s,Z_s^2,\tilde Z_s^2),\qquad s\in[0,T].
	\]
Let us define $L$ and $\tilde L$ as
\begin{align*}
	L_s&:=\frac{g(s, Z^1_s, \tilde Z^1_s) - g(s, Z^2_s, \tilde Z^1_s)}{|\Delta Z_s|^2}\, \Delta Z_s\, \mathbf{1}_{\{\Delta Z_s\neq 0\}}\\
	\tilde L_s&:=\frac{g(s, Z^2_s, \tilde Z^1_s) - g(s, Z_s^2, \tilde Z_s^2)}{ | \Delta \tilde Z_s|^2} \Delta \tilde Z_s \mathbf{1}_{\{\Delta \tilde Z_s\neq 0\}},
\end{align*}
so that
\[
	\Delta g(s,Z_s,\tilde Z_s)= L_s\Delta Z_s+\tilde L_s\Delta \tilde Z_s
\]
	Using Proposition \ref{prop:qvbound}, we consider a sequence of deterministic times $0=\tau_0\leq\tau_1\leq\cdots\leq\tau_N=T$ that slices the BMO martingales
	$Z^1\cdot W$ and $Z^2\cdot W$ into pieces with BMO norm of size $\veps$ and such that $\tau_{j+1}-\tau_j<\veps$ for all
	$j=0,1,\dots,N-1$. 
	Then, for $\tau_j\leq t\leq \tau_{j+1}$, we obtain
	\begin{equation}\label{eq:delBSDE}
		\Delta Y_t=\Delta Y_{\tau_{j+1}}+\int_t^{\tau_{j+1}} \!\int_{\tilde\Omega}\!\Delta g(s,Z_s,\tilde Z_s)\,d\tilde{\bbP}\,ds
		-\int_t^{\tau_{j+1}} \Delta Z_s\cdot dW_s.
	\end{equation}
	We define the \emph{sliced BMO norm} on $[\tau_j,\tau_{j+1}]$ by
	\[
	\big\|Z^i\cdot W\big\|_{\mathrm{BMO}_j(\bbP)}
	:=\sup_{\tau\in\scT}\Big\|\bbE_\tau\Big[\int_\tau^{\tau_{j+1}}|Z^i_s|^2\,ds\Big]\Big\|_{L^\infty}^{1/2},
	\]
	where $\scT$ denotes the collection of all stopping times valued in $[\tau_j,\tau_{j+1}]$. Similarly,
	$\Lambda:=\mathcal{E}(\int_0^\cdot\widetilde{\bbE}[L_s]\cdot dW_s)$ defines an equivalent probability measure $\bbP^L$ by the Girsanov theorem with Brownian motion $W^L_t:=W_t-\int_0^t\widetilde{\bbE}[L_s]\,ds$, we set
	\[
	\big\|Z^i\cdot W^L\big\|_{\mathrm{BMO}_j(\bbP^L)}
	:=\sup_{\tau\in\scT_j}\Big\|\bbE_\tau^{\bbP^L}\Big[\int_\tau^{\tau_{j+1}}|Z^i_s|^2\,ds\Big]\Big\|_{L^\infty}^{1/2}.
	\]
	According to Proposition \ref{Prop: estimateZ}, we have $Z^1,Z^2\in\bbH^2_{\mathrm{BMO}}$, 
	so in particular $\big(\int_0^t \tilde{\bbE}[L_s]\cdot dW_s\big)_{t\in[0,T]}$ 
	is a $\mathrm{BMO}$ martingale. Moreover, by Corollary \ref{Cor: YZ estimates} and Lemma \ref{lem:BMO-change-measure}, there exists a constant 
	$C>0$, depending only on $(T,K_z,K_{\tilde z}, K_0)$, such that
	\begin{equation}\label{eq:BMO-bound}
		\begin{aligned}
			\Big\|\int_0^\cdot \tilde{\bbE}[L_s]\cdot dW_s\Big\|_{\mathrm{BMO}}
			&+\Big\|\int_0^\cdot Z^1_s\cdot dW_s\Big\|_{\mathrm{BMO}}
			+\Big\|\int_0^\cdot Z^2_s\cdot dW_s\Big\|_{\mathrm{BMO}}
			\le C,\\
			\frac{1}{C}\norm{Z^i\cdot W^L}_{\mathrm{BMO}_{j}(\bbP^L)}&\leq\norm{Z^i\cdot W}_{\mathrm{BMO}_{j}(\bbP)}\leq C\norm{Z^i\cdot W^L}_{\mathrm{BMO}_{j}(\bbP^L)},\qquad i=1,2,\ j=0,1,\dots,N-1,
		\end{aligned}
	\end{equation}
	Under $\bbP^L$ we can rewrite \eqref{eq:delBSDE} as
	\begin{equation}\label{eq:lin-under-PL-last}
		\Delta Y_t=\Delta Y_{\tau_{j+1}}+\int_t^{\tau_{j+1}} \widetilde{\bbE}[\tilde L_s\,\Delta\tilde Z_s]\,ds
		-\int_t^{\tau_{j+1}} \Delta Z_s\cdot dW^L_s,\qquad \tau_j\le t\le \tau_{j+1}.
	\end{equation}
	By the It\^o formula, for any $t\in[\tau_j,\tau_{j+1}]$, we have
	\begin{equation}\label{itostab}
		\begin{aligned}
			|\Delta Y_t|^2+\bbE^L_t\int_t^{\tau_{j+1}}|\Delta Z_s|^2ds	&\leq \bbE^L_t|\Delta Y_{\tau_{j+1}}|^2+2\bbE^L_t\int_t^{\tau_{j+1}}|\Delta Y_s| \abs{\widetilde{\bbE}[\tilde L_s\,\Delta\tilde Z_s]}ds\\
			&\leq \bbE^L_t|\Delta Y_{\tau_{j+1}}|^2+2\norm{\sup_{s\in[t,\tau_{j+1}]}|\Delta Y_s|}_{L^\infty}\bbE^L_t\int_t^{\tau_{j+1}}\widetilde{\bbE}[\tilde L_s\,\Delta\tilde Z_s]ds.
		\end{aligned}
	\end{equation}
	Note that there exists a constant $A$ such that
	\begin{align*}
		\abs{\widetilde{\bbE}[\tilde L_s\,\Delta\tilde Z_s]}&\leq A\widetilde{\bbE}[(1+|Z^1_s|+|Z^2_s|+|\tilde Z^1_s|+|\tilde Z^2_s|)|\Delta\tilde Z_s|]\\
		&\leq A\edg{\brak{|Z^1_s|+|Z^2_s|}\bbE|\Delta Z_s|+\bbE[(1+| Z^1_s|+| Z^2_s|)|\Delta Z_s|]}.
	\end{align*}
	The Fefferman inequality for BMO martingales yields that
	{\small
		\begin{align*}
			\int_t^{\tau_{j+1}} \bbE\edg{(1+|Z_s^1|+|Z_s^2|)|\Delta Z_s|}ds
			&\leq \sqrt{2}\bbE\edg{\brak{\int_t^{\tau_{j+1}}(1+|Z_s^1|+|Z_s^2|)^2ds}^{1/2}}\norm{\Delta Z\cdot W}_{\mathrm{BMO}_{j}(\bbP)}\\
			&\leq \sqrt{6} C\brak{\veps+\norm{Z^1\cdot W}_{\mathrm{BMO}_{j}(\bbP)}+\norm{Z^2\cdot W}_{\mathrm{BMO}_{j}(\bbP)}}\norm{\Delta Z\cdot W^L}_{\mathrm{BMO}_{j}(\bbP^L)}\\
			&\leq 3\sqrt{6} C\veps\norm{\Delta Z\cdot W^L}_{\mathrm{BMO}_{j}(\bbP^L)}.
	\end{align*}}
	On the other hand,
	{\small\begin{align*}
			\bbE^L_t\int_t^{\tau_{j+1}}\brak{|Z^1_s|+|Z^2_s|}\bbE|\Delta Z_s|ds
			&\leq\sqrt{2}\bbE^L_t\brak{\int_t^{\tau_{j+1}}\brak{|Z^1_s|^2+|Z^2_s|^2}ds}^{1/2}\brak{\int_t^{\tau_{j+1}}\bbE|\Delta Z_s|^2ds}^{1/2} \\
			&\leq \sqrt{2}C\brak{\norm{Z^1\cdot W^L}_{\mathrm{BMO}_{j}(\bbP^L)}+\norm{Z^2\cdot W^L}_{\mathrm{BMO}_{j}(\bbP^L)}}\norm{\Delta Z\cdot W^L}_{\mathrm{BMO}_{j}(\bbP^L)}\\
			&\leq 2\sqrt{2} C^2\veps\norm{\Delta Z\cdot W^L}_{\mathrm{BMO}_{j}(\bbP^L)}.
	\end{align*}}
	As a result, \eqref{itostab} becomes
	{\small\begin{align*}
			|\Delta Y_t|^2+\bbE^L_t\int_t^{\tau_{j+1}}|\Delta Z_s|^2ds
			&\leq \bbE^L_t|\Delta Y_{\tau_{j+1}}|^2
			+4A(3\sqrt{6} C+2\sqrt{2} C^2)\veps\norm{\sup_{s\in[t,{\tau_{j+1}}]}|\Delta Y_s|}_{L^\infty}\norm{\Delta Z\cdot W^L}_{\mathrm{BMO}_{j}(\bbP^L)}\\
			&\leq \norm{\Delta Y_{\tau_{j+1}}}^2_{L^\infty}
			+\frac{1}{4} \brak{\norm{\sup_{s\in[{\tau_{j}},{\tau_{j+1}}]}|\Delta Y_s|}^2_{L^\infty}+\norm{\Delta Z\cdot W^L}_{\mathrm{BMO}_{j}( \bbP^L)}^2},
	\end{align*}}
	by letting $\veps$ small enough.
	Then, we have $\norm{\sup_{s\in[{\tau_{j}},{\tau_{j+1}}]}|\Delta Y_s|^2}_{L^\infty}$ and $\norm{\Delta Z\cdot W^L}_{\mathrm{BMO}_{j}(\bbP^L)}^2$ are both bounded by
	\begin{align*}
		\norm{\Delta Y_{\tau_{j+1}}}^2_{L^\infty}
		+\frac{1}{4} \norm{\sup_{s\in[{\tau_{j}},{\tau_{j+1}}]}|\Delta Y_s|}^2_{L^\infty}
		+\frac{1}{4}\norm{\Delta Z\cdot W^L}_{\mathrm{BMO}_{j}( \bbP^L)}^2.
	\end{align*}
	This implies that
	\begin{align*}
		\norm{\sup_{s\in[{\tau_{j}},{\tau_{j+1}}]}|\Delta Y_s|^2}_{L^\infty}+\norm{\Delta Z\cdot W^L}_{\mathrm{BMO}_{j}( \bbP^L)}^2\leq 4\norm{\Delta Y_{\tau_{j+1}}}^2_{L^\infty}.
	\end{align*}
	Applying the above estimates iteratively from $T$ to $0$, we obtain the estimate
	\[
	\norm{\sup_{s\in[0,T]}|\Delta Y_s|^2}_{L^\infty}+\norm{\Delta Z\cdot W^L}_{\mathrm{BMO}(\bbP^L)}^2\leq N4^N\norm{\Delta\xi}_{L^\infty}^2,
	\]
	where $N$ depends only on $(T, K_z, K_{\tilde z}, K_0)$. Another application of Lemma \ref{lem:BMO-change-measure} proves the claim.
\end{proof}

We can remove $\mathrm{(H_\xi)}$ under a more restrictive structural assumption on $g$.
\begin{theorem}\label{thm: stability2}
	For $i=1,2$, consider $(Y^i,Z^i)\in\bbS^\infty\times\bbH^2_{\mathrm{BMO}}$ solve
	\[
	Y_t^i=\xi^i+\int_t^T \!\int_{\tilde{\Omega}} g(s,Z_s^i,\tilde{Z}_s^i)\,d\tilde{\bbP}\,ds-\int_t^T Z_s^i\cdot dW_s,\qquad t\in[0,T],
	\]
	with $\xi^1,\xi^2,g$ satisfying $\Hzero,\Hqsep$, and $\Hbound$.
	Then there exists a constant $C> 0$ such that
	\[
	\|Y^1-Y^2\|_{\bbS^{\infty}}+\|Z^1-Z^2\|_{\bbH^2_{\mathrm{BMO}}}
	\ \le\ C\|\xi^1-\xi^2\|_{L^\infty},
	\]
	where $C$ depends only on $(T, K_z, K_{\tilde z}, K_0)$.	
\end{theorem}
\begin{proof}
	In the previous proof, let us take $\tau_j=\frac{jT}{N}$, where $N$ is large enough. Consider $u=\tau_j$ and $v=\tau_{j+1}$. By the It\^o formula, for any $t\in[u,v]$, we have
	\begin{align}\label{itostab2}\notag
		|\Delta Y_t|^2+\bbE^L_t\int_t^v|\Delta Z_s|^2ds	&\leq \bbE^L_t|\Delta Y_v|^2+2K_{\tilde z}\bbE^L_t\int_t^v|\Delta Y_s| \bbE\edg{(1+|Z_s^1|+|Z_s^2|)|\Delta Z_s|}ds\\\notag
		&\leq \bbE^L_t|\Delta Y_v|^2+2K_{\tilde z}\bbE^L_t\sup_{s\in[t,v]}|\Delta Y_s|\int_t^v \bbE\edg{(1+|Z_s^1|+|Z_s^2|)|\Delta Z_s|}ds\\
		&\leq \bbE^L_t|\Delta Y_v|^2+2K_{\tilde z}\brak{\bbE^L_t\sup_{s\in[u,v]}|\Delta Y_s|}\bbE\int_u^v \edg{(1+|Z_s^1|+|Z_s^2|)|\Delta Z_s|}ds.
	\end{align}
	On the other hand, by the Fefferman inequality for BMO martingales, we have
	\begin{align*}
		\bbE\int_u^v \edg{(1+|Z_s^1|+|Z_s^2|)|\Delta Z_s|}ds\leq \sqrt{2}\bbE\edg{\brak{\int_u^v(1+|Z_s^1|+|Z_s^2|)^2ds}^{1/2}}\norm{\Delta Z\cdot W}_{\mathrm{BMO}_{[u,v]}(\bbP)},
	\end{align*}
	where 
	\[
	\norm{\Delta Z\cdot W}_{\mathrm{BMO}_{[u,v]}(\bbP)}=\brak{\sup_{\tau\in\scT}\norm{\bbE_\tau\int_{\tau}^v|\Delta Z|^2ds}_{L^\infty}}^{1/2}
	\]
	and $\scT$ denotes the collection of all stopping times valued in $[u,v]$. By Lemma \ref{lem:BMO-change-measure} applied on the interval $[u,v]$ and Corollary \ref{Cor: YZ estimates}, there exists a constant $C$, which only depends on $(T,K_z,K_{\tilde z}, K_0)$, such that
	\begin{align*}
		\norm{\Delta Z\cdot W}_{\mathrm{BMO}_{[u,v]}(\bbP)}\leq \frac{C}{2\sqrt{2}K_{\tilde z}}\norm{\Delta Z\cdot W^L}_{\mathrm{BMO}_{[u,v]}(\tilde \bbP)}.
	\end{align*}
	In addition, since $t\mapsto \brak{\bbE\int_0^t(1+|Z_s^1|+|Z_s^2|)^2ds}^{1/2}$
	is uniformly continuous, for any $\veps>0$,  there exists $\delta>0$ such that $v-u<\delta$ implies
	\begin{align}\label{unifcont}
		\brak{\bbE\int_u^v(1+|Z_s^1|+|Z_s^2|)^2ds}^{1/2}<\frac{\veps}{C}.
	\end{align}
	As a result, \eqref{itostab2} becomes
	\begin{align*}
		|\Delta Y_t|^2+\bbE^L_t\int_t^v|\Delta Z_s|^2ds&\leq\bbE^L_t|\Delta Y_v|^2+\veps\bbE^L_t\sup_{s\in[u,v]}|\Delta Y_s|\norm{\Delta Z\cdot W^L}_{\mathrm{BMO}_{[u,v]}(\tilde \bbP)}\\
		&\leq \norm{\Delta Y_v}^2_{L^\infty}+\frac{\veps}{2} \norm{\sup_{s\in[u,v]}|\Delta Y_s|}^2_{L^\infty}+\frac{\veps}{2}\norm{\Delta Z\cdot W^L}_{\mathrm{BMO}_{[u,v]}(\tilde \bbP)}^2.
	\end{align*}
	By taking $\veps =1/2$, by a similar argument in the proof of the previous theorem, we have
	\begin{align*}
		\norm{\sup_{s\in[u,v]}|\Delta Y_s|^2}_{L^\infty}+\norm{\Delta Z\cdot W^L}_{\mathrm{BMO}_{[u,v]}(\tilde \bbP)}^2\leq 4\norm{\Delta Y_v}^2_{L^\infty}.
	\end{align*}
	If we partition $[0,T]$ into intervals of size less than $\delta$ and apply the above estimates iteratively from $T$ to $0$, we obtain the estimate
	\[
	\norm{\sup_{s\in[0,T]}|\Delta Y_s|^2}_{L^\infty}+\norm{\Delta Z\cdot W^L}_{\mathrm{BMO}(\tilde \bbP)}^2\leq N4^N\norm{\Delta\xi}_{L^\infty}^2,
	\]
	where $N$ is the number of subintervals in the partition. Another application of Lemma \ref{lem:BMO-change-measure} proves the claim.
\end{proof}
\section{Proof: existence and uniqueness} \label{Sec: ExistenceUniqueness}
In this section, we prove the existence and uniqueness results for the mean-field BSDE \eqref{mainBSDE} in the two quadratic regimes introduced earlier. Under Assumptions ${\Hzero}$ and ${\Hqunif}$, we establish Theorem \ref{thm: mainthm2} for terminal conditions of the form $\xi=\phi(W_{t_1},\ldots,W_{t_n})$ with $\phi$ bounded and $\alpha$-H\"older continuous. The generalization to Theorem \ref{thm:mainHolder} can be achieved by measure-change and time-change. Under Assumptions ${\Hzero}$ and ${\Hqsep}$, we establish Theorem \ref{thm: mainthm3} for terminal conditions of the form $\xi=\phi(W_{[0,T]})$ with $\phi$ bounded and uniformly continuous in the sup norm.

The proofs follow the same general scheme, but with different approximation methods adapted to the corresponding terminal regularity. We first prove in Lemma \ref{Lem: Malliavin} that if the terminal condition $\xi$ is Malliavin differentiable with uniformly bounded derivative, then \eqref{mainBSDE} admits a solution in $\bbS^\infty\times\bbS^\infty$. The argument is based on truncating the quadratic generator to obtain a Lipschitz generator.

Next, we approximate the desired terminal condition by a sequence $(\xi^n)$ with bounded Malliavin derivatives. In the unified case $\Hqunif$, we use mollification of the H\"older function $\phi$ on $(\bbR^d)^n$. In the separable case $\Hqsep$, we use the Moreau--Yosida approximation on path space, which yields Lipschitz functionals with bounded Malliavin derivatives. For each $n$, Lemma \ref{Lem: Malliavin} provides a solution \((Y^n,Z^n)\), and the corresponding stability theorem (Theorem \ref{thm: stability} under $\Hqunif$\ or Theorem \ref{thm: stability2} under $\Hqsep$) implies that \((Y^n,Z^n)\) is Cauchy in the Banach space $\bbS^\infty\times\bbH^2_{\mathrm{BMO}}$. Passing to the limit yields a solution to \eqref{mainBSDE}, and uniqueness follows from the same stability result.

\begin{lemma}\label{Lem: Malliavin}
	Assume $\Hzero$ and ${\Hqunif}$ hold, and suppose $\xi$ is Malliavin differentiable with bounded derivative
	$\norm{D\xi}_{\bbH^\infty}<\infty$.
	Then the mean-field BSDE admits a solution in $\bbS^\infty\times\bbS^\infty$.
\end{lemma}

\begin{proof}
	For $k\in\bbN$, define the truncation map $\pi_k:\bbR^d\to\bbR^d$ by
	\[
	\pi_k(z):=\frac{|z|\wedge k}{|z|}\,z\,\mathbf{1}_{\{z\neq 0\}},\qquad \pi_k(0):=0,
	\]
	and define a sequence of drivers $(\bar g_k)_{k\in\bbN}$ by
	\[
	\bar g_k(t,z,\tilde z):=g\big(t,\pi_k(z),\pi_k(\tilde z)\big),\qquad (t,z,\tilde z)\in[0,T]\times\bbR^d\times\bbR^d.
	\]
	Note that $|\pi_k(z)|\le k$ for all $z$, and $\pi_k$ is $1$--Lipschitz:
	\begin{equation}\label{eq:pi-Lip}
		|\pi_k(z)-\pi_k(z')|\le |z-z'|,\qquad \forall z,z'\in\bbR^d.
	\end{equation}
	We claim that $\bar g_k$ is globally Lipschitz, we have for all $(t,z,\tilde z,z',\tilde z')$,
	\begin{align*}
		|\bar g_k(t,z,\tilde z)-\bar g_k(t,z',\tilde z')|
		&\le K_z\Big(1+|\pi_k(z)|+|\pi_k(z')|+|\pi_k(\tilde z)|+|\pi_k(\tilde z')|\Big)\,|\pi_k(z)-\pi_k(z')| \\
		&\quad + K_{\tilde z}\Big(1+|\pi_k(z)|+|\pi_k(z')|+|\pi_k(\tilde z)|+|\pi_k(\tilde z')|\Big)\,|\pi_k(\tilde z)-\pi_k(\tilde z')|.
	\end{align*}
	Since $|\pi_k(\cdot)|\le k$, the common factor is bounded by $1+4k$, and by \eqref{eq:pi-Lip},
	\begin{equation}\label{eq:gk-Lip}
		|\bar g_k(t,z,\tilde z)-\bar g_k(t,z',\tilde z')|
		\le (1+4k)\,K_z\,|z-z'| + (1+4k)\,K_{\tilde z}\,|\tilde z-\tilde z'|.
	\end{equation}
	Thus, $\bar g_k$ satisfies $\Hzero$ and $\Hlip$. Consequently, the mean-field BSDE with driver $\bar g_k$,
	\begin{equation}\label{eq: lipMFBSDE}
		Y_t^k = \xi + \int_t^T \int_{\tilde{\Omega}} \bar{g}_k(s, Z_s^k, \tilde Z_s^k) \, d\tilde\bbP \, ds - \int_t^T Z_s^k \cdot dW_s,
		\quad 0 \leq t \leq T,
	\end{equation}
	admits a unique solution $(Y^k,Z^k)\in\bbS^2\times\bbH^2$ (see Theorem 3.1 in \cite{Buckdahn2009MFBSDE}).
	Moreover, Proposition \ref{prop: lipbound} ensures that for each $k\in\bbN$, $(Y^k,Z^k)\in\bbS^\infty\times\bbH^\infty$.
	In particular, there exists a constant $C_Z>0$ such that
	\[
	\|Z^k\|_{\bbH^\infty}\le C_Z,\qquad \forall k\in\bbN.
	\]
	Since $(\tilde Y^k,\tilde Z^k)$ are independent copies of $(Y^k,Z^k)$, the same bound holds for $\tilde Z^k$.
	Choose $k^*\in\bbN$ such that $k^*>C_Z$. Then for this $k^*$ the truncation is inactive along the solution and therefore
	\[
	\bar g_{k^*}(s,Z_s^{k^*},\tilde Z_s^{k^*})
	= g(s,Z_s^{k^*},\tilde Z_s^{k^*}),\qquad 0\le s\le T.
	\]
	Hence, $(Y^{k^*},Z^{k^*})$ solves the original mean-field BSDE \eqref{mainBSDE}, which completes the proof.
\end{proof}
We are now in a position to state the main result of this section. For the reader’s convenience, we first restate the main theorem below in a simplified setting and then generalize the statement. To complete the proof of the following theorem, it remains to extend the existence result from Malliavin differentiable terminal conditions to the general bounded case. Uniqueness then follows directly from the stability result established in Theorem \ref{thm: stability}.

\begin{theorem}\label{thm: mainthm2}
	Let Assumptions ${\Hzero}, {\Hqunif}$ and $({\rm H_{\xi}})$ hold.
	Then, there exists a unique solution \((Y,Z) \in \bbS^\infty \times \bbH^2_{\mathrm{BMO}}\) to the BSDE \eqref{mainBSDE}.
\end{theorem}
\begin{proof}
	Let $\eta_m$ be a mollifier on $(\bbR^d)^n$ with compact support $([-\frac{1}{m},\frac{1}{m}]^d)^n$. Consider $\phi^m(x):=\eta_m*\phi$, which is Lipschitz. Then $\xi^m:=\phi^m(W_{t_1},W_{t_2},\cdots, W_{t_n})$ has bounded Malliavin derivatives, and Lemma \ref{Lem: Malliavin} yields a corresponding solution $(Y^m,Z^m)$. Since $\phi^m$ has the same H\"older constant as $\phi$ for any $m$, Theorem \ref{thm: stability} provides the convergence of $(Y^m, Z^m)$ to $(Y,Z)$ in $\bbS^\infty\times\bbH^2_{\mathrm{BMO}}$. Uniqueness follows from Theorem \ref{thm: stability} as well.
\end{proof}
\begin{theorem}\label{thm: mainthm3}
	Let Assumptions ${\Hzero}, \Hqsep$ and ${\Hbound}$ hold. In addition, assume that there exists a function $\phi:C([0,T];\bbR^d)\to\bbR$ such that $\xi=\phi(W_{[0,T]})$ and uniformly continuous in sup norm.
	Then, there exists a unique solution \((Y,Z) \in \bbS^\infty \times \bbH^2_{\mathrm{BMO}}\) to the BSDE \eqref{mainBSDE}.
\end{theorem}
\begin{proof}
	Let us consider the Moreau-Yosida approximation of $\phi$:
	\begin{align*}
		\phi^n(\omega):=\inf_{y\in C([0,T];\bbR^d)}\crl{\phi(y)+n\sup_{t\in[0,T]}\abs{y(t)-\omega(t)}}.
	\end{align*}
	Since
	$\inf_{y\in C([0,T];\bbR^d)}\phi(y)\leq \phi^n(\omega)\leq \phi(\omega),$
	$\phi^n$ is bounded. In addition, note that, for any $\omega,\omega'$
	{\small\begin{align*}
			\phi^n(\omega)&=\inf\crl{\phi(y)+n\sup_{t\in[0,T]}\abs{y(t)-\omega(t)}}\leq \inf\crl{\phi(y)+n\sup_{t\in[0,T]}\abs{y(t)-\omega'(t)}}+n\sup_{t\in[0,T]}\abs{\omega'(t)-\omega(t)}\\
			&\leq \phi^n(\omega')+n\sup_{t\in[0,T]}\abs{\omega'(t)-\omega(t)}.
	\end{align*}}
	Therefore, $\phi^n$ is Lipschitz in the sup norm. Lastly, since $\phi$ is uniformly continuous, for $\veps>0$, there exists $N$ such that $\sup_{t\in[0,T]}\abs{y(t)-\omega(t)}<\frac{\veps}{2n}$ implies $|\phi(y)-\phi(\omega)|\leq \veps/2$ for all $n\geq N$. Then,
	\begin{align*}
		|\phi^n(\omega)-\phi(\omega)|&=\abs{\inf_{y\in C([0,T];\bbR^d)}\crl{\phi(y)-\phi(\omega)+n\sup_{t\in[0,T]}\abs{y(t)-\omega(t)}}}<\veps.
	\end{align*}
	Therefore,  $\phi^n(\omega)$ converges to $\phi(\omega)$ uniformly in $\omega$. By Proposition 3.2 of \cite{cheridito2014bsdes}, $\xi^n:=\phi^n(W)$ has bounded Malliavin derivatives. By Lemma \ref{Lem: Malliavin}, each approximating BSDE with terminal condition \(\xi^n\) admits a solution \((Y^n, Z^n) \in \bbS^\infty \times \bbS^\infty\), satisfying
	\begin{equation}\label{eq: limitBSDE}
		Y_t^n = \xi^n + \int_t^T \int_{\tilde{\Omega}} g(s, Z_s^n, \tilde{Z}_s^n) \, d\tilde{\bbP} \, ds - \int_t^T Z_s^n \cdot dW_s.
	\end{equation}
	Moreover, since $\phi^n(\omega)$ converges to $\phi(\omega)$ uniformly in $\omega$, we have $\norm{\xi^n-\xi}_{L^\infty}\to0$. Then, by Theorem \ref{thm: stability2}, we obtain the convergence of $(Y^n, Z^n)$ to $(Y,Z)$ in $\bbS^\infty\times\bbH^2_{BMO}$. The uniqueness follows again from Theorem \ref{thm: stability2}.
\end{proof}
Now let us generalize the above theorem to the version we stated in Section \ref{Sec: Mainresult}.
\begin{theorem}
	Assume that there is a constant $C$ such that $b:[0,T]\times\bbR^k\to\bbR^k$ and $\sigma:[0,T]\to \bbR^{k\times n}$ are measurable functions with 
	\begin{itemize}
		\item 	$|b(t,x)-b(t,x')|\leq C|x-x'|, |b(t,0)|\leq C$ for every $(t,x)\in[0,T]\times\bbR^k$, and
		\item $C^{-1}|v|^2\leq v^\intercal(\sigma\sigma^\intercal)(t)v\leq C|v|^2$ for all $v\in\bbR^k$ and $t\in[0,T]$.
	\end{itemize} Then, let $X$ be the strong solution of 
	\[
	dX_t=b(t, X_t)dt+\sigma(t)dW_t;\qquad X_0=x_0.
	\]
	Then, we can replace $W$ to $X$ in Theorem \ref{thm: mainthm2} and Theorem \ref{thm: mainthm3}.
\end{theorem}
\begin{proof}
	First, let us consider Theorem \ref{thm: mainthm2}. 
	First, note that $D_t X_s$ is bounded by Lemma 4.2 in \cite{cheridito2014bsdes}. Therefore, in the proof of Theorem \ref{thm: mainthm2},
	\[
	\xi^m:=\phi^m(X_{t_1},X_{t_2},\cdots,X_{t_n})
	\]
	has bounded Malliavin derivatives, which give the corresponding solution $(Y^m, Z^m)$. 
	
	Let us define $\bbQ$ by
	\[
	\left.\frac{d\bbQ}{d\bbP}\right|_t=\scE_t\brak{\int_0^\cdot -(\sigma^\intercal (\sigma\sigma^\intercal)^{-1}b)(t,X_t)^\intercal dW_t}.
	\]
	Then we can apply the Girsanov theorem to conclude that $\bbQ$ is equivalent to $\bbP$ and that \[
	B_t=W_t+\int_0^t\brak{\sigma^\intercal (\sigma\sigma^\intercal)^{-1}b}(s,X_s)ds
	\] is $\bbQ$-Brownian motion by the \cite{klebaner2014stochastic}. After this measure change we have $dX_t=\sigma(t)dB_t$ and the argument in Proposition \ref{prop:qvbound} works under the changed measure, except the expectation calculation in \eqref{eq:holdercalc}. This calculation can be done by the time change.
	Since the operator norm of $\sigma\sigma^\intercal$ is bounded away from $0$, under an appropriate time change, there exists a linear function $\Lambda$ such that 
	\[
	\xi^m=\phi^m(X_{t_1},X_{t_2},\cdots,X_{t_n})=(\phi^m\circ\Lambda)(\tilde B_{s_1},\tilde B_{s_2},\cdots,\tilde B_{s_{l}})
	\]
with $s_n<\infty$ for another Brownian motion $\tilde B$. Then, we can use the same argument in Proposition \ref{prop:qvbound} to conclude $\norm{\int_0^T|Z_s|^2ds}_{L^\infty}\leq C_\alpha T^\alpha$ and deterministic BMO slicing. Then, the same argument for Theorem \ref{thm: stability} yields existence and uniqueness.
	
	On the other hand, let us consider Theorem \ref{thm: mainthm3}. Note that we only need to prove that, for Lipschitz $\phi:C([0,T];\bbR)\to\bbR$, the Malliavin derivative of $\phi(X_{[0,T]})$ is bounded: if this statement is true, $\phi^n$ in the proof of Theorem \ref{thm: mainthm3} has bounded Malliavin derivatives, so we can use the same argument to prove the existence and uniqueness.
	
	Note that, if $\phi$ is Lipschitz, for linear interpolator $l^m:\bbR^m\to C([0,T];\bbR)$ and $t_i=\frac{iT}{m}$, we have
	\begin{align*}
		\bbE\edg{\abs{\phi(X_{[0,T]})-(\phi\circ l^m)(X_{t_1},X_{t_2},\cdots,X_{t_m})}^p}&\leq C^2\bbE\edg{\max_{i=1,2,\cdots,m}\sup_{t\in(t_i,t_{i+1}]}|X_t-X_{t_i}|^p},
	\end{align*}
	for $p>2$.
	Note that we have
	\begin{align*}
		\sup_{t\in(t_i,t_{i+1}]}|X_t-X_{t_i}|^p&\leq\sup_{t\in(t_i,t_{i+1}]}\abs{\int_{t_i}^tb(s,X_s)ds+\int_{t_i}^t\sigma(s)dW_s}^p\\
		&\leq2^{p-1}\sup_{t\in(t_i,t_{i+1}]}\brak{\abs{\int_{t_i}^tb(s,X_s)ds}^p+\abs{\int_{t_i}^t\sigma(s)dW_s}^p}\\
		&\leq 2^{p-1}\brak{(t_{i+1}-t_i)^{p-1}\int_{t_i}^{t_{i+1}} |b(s,X_s)|^pds+\sup_{t\in(t_i,t_{i+1}]}\abs{\int_{t_i}^t\sigma(s)dW_s}^p}\\
		&\leq 2^{p-1}\brak{(2T)^{p-1}m^{-(p-1)}C^p\int_{0}^{T} (1+|X_s|^p)ds+\sup_{t\in(t_i,t_{i+1}]}\abs{\int_{t_i}^t\sigma(s)dW_s}^p}.
	\end{align*}
	It is easy to see that, as $m\to\infty$,
	\[
	\bbE\max_{i=1,2,\cdots,m}m^{-(p-1)}\int_{0}^{T} (1+|X_s|^p)ds\leq \bbE m^{-(p-2)}\int_{0}^{T} (1+|X_s|^p)ds\to 0
	\]
	because $X\in\bbS^p$ from our assumption. On the other hand
	\begin{align*}
		\bbE\max_{i=1,2,\cdots,m}\sup_{t\in(t_i,t_{i+1}]}\abs{\int_{t_i}^t\sigma(s)dW_s}^p&\leq \sum_{i=1}^m\bbE\sup_{t\in(t_i,t_{i+1}]}\abs{\int_{t_i}^t\sigma(s)dW_s}^p
		\leq C_p\sum_{i=1}^m\bbE\abs{\int_{t_i}^{t_{i+1}}|\sigma(s)|^2ds}^{p/2}\\
		&\leq C_p \sum_{i=1}^m (t_{i+1}-t_i)^{p/2-1}\int_{t_i}^{t_{i+1}}|\sigma(s)|^pds\\
		&\leq C_p\brak{\frac{T}{m}}^{p/2-1}\int_0^T|\sigma(s)|^pds\to0,
	\end{align*}
	as $m\to\infty$. Here, $C_p$ is the constant from the Burkholder-Davis-Gundy inequality.
	As a result, $(\phi\circ l^m)(X_{t_1},X_{t_2},\cdots,X_{t_m})\to\phi(X_{[0,T]})$ in $L^p$.
	In addition, by \cite[Proposition 1.2.4]{nualart2006malliavin} and the boundedness of $D_tX_s$, we know that $(\phi\circ l^m)(X_{t_1},X_{t_2},\cdots,X_{t_m})$ has bounded Malliavin derivative.
	Therefore, by \cite[Lemma 1.2.3]{nualart2006malliavin}, we conclude that $\phi(X_{[0,T]})$ has bounded Malliavin derivative whenever $\phi$ is Lipschitz.
\end{proof}
\section{Quadratic MFBSDEs with Small Terminal Oscillation and Utility Maximization}\label{Sec: application}
In this section, we extend the foundational work of \cite{hu2005utility} to the mean-field case. Specifically, we explore how a trader's decision-making process evolves in a financial market influenced by collective behavior. Unlike previous studies that assume fixed, exogenous constraints with respect to the optimization problem, we consider an endogenous and dynamic setting in which a trader's strategy is shaped not only by individual preferences but also by others' actions. This interaction leads to a more flexible model, in which part of the trading constraints are determined by market conditions rather than imposed arbitrarily.

Section \ref{subsec: appmodel} first specifies the complete financial market and introduces the equivalent portfolio and volatility-scaled representations. We then formulate the adaptive constraint and the exponential utility maximization problem for a fixed deterministic signal \(m\). Under the a priori assumption that the associated BSDE admits a sufficiently integrable solution, we impose the consistency condition \(m_t=\bbE[Z_t]\) and use the verification argument of \cite{hu2005utility} to identify the BSDE driver, value representation, and optimal strategy.

The remaining part of the section makes this derivation rigorous. In contrast to Sections \ref{Sec: Stability}--\ref{Sec: ExistenceUniqueness}, we do not impose the structural regularity assumptions used there on the terminal condition. Instead, we establish a small-terminal-oscillation result. We first establish a unique solution in a suitable small solution class for a more general copy-space BSDE by a local contraction argument followed by backward concatenation. We then transfer this result to the expectation-dependent BSDE with driver \(f(t,Z_t,\bbE[Z_t])\), and finally apply it to the utility driver derived above. This yields the well-posedness required to complete the verification argument and identify the value function and an optimal strategy.

\subsection{Financial market model setup}\label{subsec: appmodel}

Consider a financial market composed of \(d\) risky stocks. In this application, we work in the complete-market case driven by a \(d\)-dimensional Brownian motion. There are a countably infinite number of homogeneous investors who trade dynamically between the risk-free bond and the risky assets. We take the risk-free bond as numeraire and work throughout with discounted stock prices and discounted wealth, so that \(b_t\) represents the vector of excess returns. An investor, a representative of homogeneous investors, uses the following model to maximize the utility. The price process of the \(i\)-th stock, denoted by \(S^i\), evolves according to the following stochastic differential equation:
\[
\frac{\mathrm{d} S_t^i}{S_t^i} = b_t^i \, \mathrm{d} t + \sigma_t^i \, \mathrm{d} W_t, \quad i = 1, \dots, d,
\]
where \(b^i:[0,T]\to\mathbb{R}\) and \(\sigma^i:[0,T]\to\mathbb{R}^d\) are deterministic bounded measurable functions. The volatility matrix \(\sigma_t = (\sigma_t^1, \dots, \sigma_t^d)^\top\) is assumed to be invertible, and \(\sigma_t \sigma_t^\top\) is uniformly elliptic uniform in time. Thus, the market price of risk \(\theta_t\) is the deterministic bounded function
\[
\theta_t := \sigma_t^\top (\sigma_t \sigma_t^\top)^{-1} b_t, \quad t \in [0, T],
\]
where \(\theta_t\) is \(\mathbb{R}^d\)-valued.

For \( 1 \leq i \leq d \), let \( \pi^i_t \) represent the amount of money invested in stock \( i \) at time \( t \), meaning the number of shares held in stock \( i \) is \( \frac{\pi^i_t}{S^i_t} \). A predictable process \( \pi = (\pi_t^1, \dots, \pi_t^d)_{0 \leq t \leq T} \in \mathbb{R}^d \), or simply \( (\pi_t)_{0 \leq t \leq T} \), is called a trading strategy if \( \int_0^\cdot \pi_u \frac{\diff S_u}{S_u} \) is well-defined, which holds if \( \int_0^T \|\pi_t \sigma_t\|^2 \diff t < \infty \) almost surely. The corresponding wealth process \( X^\pi \), representing the investor's wealth over time until $T$, is given by:
\[
X^\pi_t = x + \sum_{i=1}^d \int_0^t \frac{\pi^i_u}{S^i_u} \diff S^i_u = x + \int_0^t \pi_u \sigma_u (\diff W_u + \theta_u \diff u),
\]
where \( x \in \mathbb{R} \) is the initial capital.

It will be convenient to use two equivalent representations of the portfolio. The original strategy \(\pi_t\) is expressed in portfolio coordinates. Since \(\sigma_t\) is invertible, define the volatility-scaled strategy
\[
p_t:=\pi_t\sigma_t,
\qquad
\pi_t=p_t\sigma_t^{-1}.
\]
Writing \(X^p:=X^\pi\) for the same wealth process under this reparametrization, we have
\[
X_t^p=x+\int_0^t p_u(\diff W_u+\theta_u\diff u).
\]
The volatility-scaled representation absorbs the diffusion matrix into the control and is therefore the natural coordinate system for the BSDE formulation below.

At maturity \(T\), the investor's terminal net wealth is \(X_T^\pi-F=X_T^p-F\), where \(F\) is an \(\mathcal F_T\)-measurable random liability; \(F\geq0\) represents a payment and \(F\leq0\) represents income. The investor's preferences are modeled by an exponential utility function applied to the terminal net wealth.
The exponential utility function is given by:
\[
U(x) = -\exp(-\alpha x), \quad x \in \mathbb{R}.
\]

For investors other than the representative, we assume that the risk aversion parameter $\alpha>0$ and the market coefficients \(b\) and \(\sigma\) are identical. To represent population averages, the terminal liability \(F\), the BSDE control process \(Z\), and other representative-agent quantities may be realized as independent copies on duplicated probability spaces. This copied-space construction is used only to represent the corresponding population distributions and expectations.

\subsection{Constraint set and aggregated BSDE signal}

We next specify the set of admissible positions. Throughout this section, we use two equivalent coordinate systems. The trading strategy and the adaptive constraint are naturally formulated in \emph{portfolio coordinates}, whereas the BSDE is naturally expressed in the \emph{volatility-scaled coordinates} introduced above. We use hats for quantities represented in portfolio coordinates and omit hats for their volatility-scaled counterparts.

To avoid defining the admissible set in terms of a BSDE solution that has not yet been introduced, we first fix a deterministic process \(m\in L^2([0,T];\mathbb{R}^d)\) in volatility-scaled coordinates and define its portfolio-coordinate representation by
\[
\hat m_t:=m_t\sigma_t^{-1},
\qquad\text{equivalently,}\qquad
m_t=\hat m_t\sigma_t.
\]
After introducing the BSDE, we will impose the consistency condition \(m_t=\mathbb{E}[Z_t]\). The process \(m\), together with its equivalent representation \(\hat m\), is then referred to as the \emph{aggregated BSDE signal}. To distinguish the pointwise and mean-field arguments of the driver, we write \(z\in\mathbb{R}^d\) for a pointwise value and \(m\in\mathbb{R}^d\) for its mean-field argument.

\begin{definition}[Mean-field adaptive constraint set]\label{def: mf_consts}
	Let \(\hat D\) be a nonempty compact subset of \(\mathbb{R}^d\). For \(\hat m\in\mathbb{R}^d\), define
	\[
	\hat C(\hat m):=\hat D\odot L(\hat m)
	:=\{a\odot L(\hat m):a\in\hat D\},
	\qquad
	L(\hat m):=\gamma\odot\hat m+\ell,
	\]
	where \(\odot\) denotes component-wise multiplication and \(\gamma,\ell\in\mathbb{R}^d\) are constant vectors. Here, \(\gamma\) determines the sensitivity of the constraint to the aggregated signal, while \(\ell\) gives the baseline component of the affine map \(L\). In volatility-scaled coordinates, the corresponding constraint map is
	\begin{equation}\label{eq:constraint-map}
		C_t(m):=\hat C(m\sigma_t^{-1})\sigma_t,
		\qquad (t,m)\in[0,T]\times\mathbb{R}^d.
	\end{equation}
\end{definition}

The two representations used below are therefore related by
\[
p_t=\pi_t\sigma_t,
\qquad
m_t=\hat m_t\sigma_t,
\qquad
C_t(m_t)=\hat C(\hat m_t)\sigma_t.
\]
Thus, for a fixed signal \(m\), the portfolio constraint is \(\hat C(\hat m_t)\), while the volatility-scaled strategy is constrained to \(C_t(m_t)\).

\cite{hu2005utility} allow a general closed underlying portfolio constraint because it is fixed with respect to the mean-field signal, whereas compactness is imposed here to control its dependence on the aggregated signal.

\begin{example}\label{ex:adaptive-box}
	If \(\hat D=\prod_{i=1}^d[-a_i,a_i]\) with \(a_i>0\), then the \(i\)-th portfolio position is constrained by
	\[
	|\pi_t^i|\leq a_i\left|\gamma_i\hat m_t^i+\ell_i\right|.
	\]
	This gives an explicit position-limit model in which the width of the feasible interval responds to the corresponding component of the signal. Under the consistency condition introduced below, this signal is determined by the BSDE control.
	
	In the scalar case \(d=1\), with \(\sigma_t=1\) and \(\hat D=[-a,a]\), the constraint set and its distance function are
	\[
	C_t(m)=[-a|\gamma m+\ell|,a|\gamma m+\ell|],
	\qquad
	\dis^2(x,C_t(m))
	=\bigl(|x|-a|\gamma m+\ell|\bigr)_+^2,
	\]
	where \(r_+:=\max\{r,0\}\) and \(x\in\mathbb{R}\) is a dummy point in the distance function. In the BSDE driver below, this point is taken to be \(x=z+\theta_t/\alpha\). This makes explicit the distinct roles played by the pointwise variable and the aggregated signal \(m\).
\end{example}

\begin{definition}[Admissible strategies for a fixed signal]
	For a deterministic signal \(m\in L^2([0,T];\mathbb{R}^d)\), the set \(\hat{\scA}(m)\) consists of all \(\mathbb{R}^d\)-valued predictable processes \(\pi\) satisfying \(\int_0^T|\pi_t\sigma_t|^2dt<\infty\) almost surely, \(\pi_t\in\hat C(\hat m_t)\) almost surely for almost every \(t\in[0,T]\), and
	\[
	\left\{ \exp(-\alpha X^\pi_\tau) : \tau \text{ is a stopping time taking values in } [0, T] \right\}
	\]
	forms a uniformly integrable family.
\end{definition}

For a fixed signal \(m\), the investor’s objective is to choose an admissible self-financing trading strategy to maximize the expected utility of the net wealth at maturity \(T\):
\[
V^m(0, x) := \sup_{\pi\in \hat{\scA}(m)} \mathbb{E} \left[ -\exp\left( -\alpha \left( x + \int_0^T \pi_u \frac{dS_u}{S_u} - F \right) \right) \right],
\]
where \(V^m(0,\cdot)\) is the value function associated with the fixed signal \(m\). Since \(\sigma_t\) is invertible, the portfolio and volatility-scaled formulations introduced above are equivalent.
\begin{lemma}\label{Lem: conset estimate}
	For every \((t,m)\in[0,T]\times\mathbb{R}^d\), the set \(C_t(m)\) is nonempty and compact. Moreover, there exists a constant \(K_1\geq0\) such that
	\[
	\sup_{c\in C_t(m)}|c|\leq K_1(1+|m|)
	\]
	and
	\[
	\left|\dis(x,C_t(m))-\dis(x',C_t(m'))\right|
	\leq |x-x'|+K_1|m-m'|,
	\]
	for all \(t\in[0,T]\) and \(x,x',m,m'\in\mathbb{R}^d\).
\end{lemma}

\begin{proof}
	Compactness follows because \(C_t(m)\) is the continuous image of the compact set \(\hat D\). The uniform ellipticity and boundedness assumptions imply that \(\sigma\) and \(\sigma^{-1}\) are uniformly bounded. Hence, the first estimate follows directly from \eqref{eq:constraint-map}, the boundedness of \(\hat D\), and the linearity of \(L\). Moreover, for each point in \(C_t(m)\), the point obtained by using the same element of \(\hat D\) in the definition of \(C_t(m')\) is at distance at most \(K_1|m-m'|\). Combining this observation with the fact that the distance from a fixed closed set is Lipschitz gives the second estimate.
\end{proof}

Equivalently, in volatility-scaled coordinates, \(\scA(m)\) consists of all \(\mathbb{R}^d\)-valued predictable processes \(p\) satisfying \(\int_0^T|p_t|^2dt<\infty\) almost surely and \(p_t\in C_t(m_t)\) almost surely for almost every \(t\in[0,T]\). Additionally, the family
\[
\left\{ \exp(-\alpha X^p_\tau) : \tau \text{ is a stopping time taking values in } [0, T] \right\}
\]
is uniformly integrable.

The problem is therefore
\begin{equation*}
	V^m(0,x)
	=\sup_{p\in\scA(m)}\bbE\left[-\exp\left( -\alpha \left( x + \int_0^T p_t (\diff W_t + \theta_t \diff t) - F \right) \right)\right].
\end{equation*}

It remains to determine the signal consistently. For a fixed deterministic signal \(m\), consider the BSDE
\[
Y_t^m
=F+\int_t^T f(s,Z_s^m,m_s)\,ds
-\int_t^T Z_s^m\cdot dW_s.
\]
The mean-field consistency condition is
\[
m_t=\mathbb{E}[Z_t^m].
\]
Substituting this identity into the fixed-signal BSDE and suppressing the superscript \(m\), we obtain
\begin{equation}\label{BSDE: application}
	Y_t = F + \int_t^T f(s,Z_s, \mathbb{E}[Z_s]) \, ds - \int_t^T Z_s  \!\cdot dW_s, \quad t \in [0, T].
\end{equation}
For a candidate solution of BSDE \eqref{BSDE: application}, set
\[
m_t:=\mathbb{E}[Z_t].
\]
The BSDE control process \(Z_t\) is expressed in the same volatility-scaled coordinates as \(p_t\). Its portfolio-coordinate representation is therefore
\[
\hat Z_t:=Z_t\sigma_t^{-1}.
\]
Since \(\sigma\) is deterministic,
\[
\hat m_t=m_t\sigma_t^{-1}=\mathbb{E}[\hat Z_t].
\]
From now on, we write
\[
\hat{\scA}:=\hat{\scA}(m),
\qquad
\scA:=\scA(m),
\qquad
V(0,x):=V^m(0,x).
\] To identify the appropriate BSDE driver, suppose for the moment that BSDE \eqref{BSDE: application} admits a sufficiently integrable solution \((Y,Z)\). We introduce the auxiliary process
\begin{equation*}
	R^{(p)}(t, x) := - \exp(-\alpha (X_t^p - Y_t)), \quad t\in[0,T], \, p\in \scA.
\end{equation*}
We use \(R^{(p)}\) for this strategy-dependent verification process.
At this stage, the calculation is used only to identify the drift condition that the driver should satisfy. The existence and integrability properties of \((Y,Z)\), and the admissibility of the candidate optimizer, will be established below.

In this construction, \(Z\) is the martingale integrand in BSDE \eqref{BSDE: application}. Since the constraint set is determined by the consistent signal \(m_t=\mathbb{E}[Z_t]\), the BSDE driver also depends on the same signal. We now specify the driver \(f\). By It\^o's formula,
\begin{equation*}
	\diff R^{(p)}(t,x)
	=R^{(p)}(t,x)\left[
	\alpha(Z_t-p_t)\cdot\diff W_t
	+\alpha v(t,p_t,Z_t)\diff t\right],
\end{equation*}
where
\begin{equation*}
	v(t,p_t,Z_t)
	=-f(t,Z_t,\bbE[Z_t])-p_t\theta_t
	+\frac{\alpha}{2}\abs{p_t-Z_t}^2.
\end{equation*}
Since \(R^{(p)}\) is negative, its finite-variation term is non-positive whenever \(v(t,p_t,Z_t)\geq0\). We therefore look for a driver for which \(v(t,p_t,Z_t)\geq0\) for all \(p\in\scA\), with pointwise equality at a minimizer of the distance term. This is obtained by choosing
\begin{equation}\label{eq: driverf}
	f(t,z,m) = \frac{\alpha}{2} \dis^2\left(z + \frac{\theta_t}{\alpha}, C_t(m) \right) - z \theta_t - \frac{\abs{\theta_t}^2}{2\alpha},
	\qquad (z,m)\in\mathbb{R}^d\times\mathbb{R}^d.
\end{equation}
Then, with \(m=\mathbb{E}[Z_t]\), we have
\begin{align*}
	v(t,p_t,Z_t)
	&=-f(t,Z_t,\bbE[Z_t])
	+\frac{\alpha}{2}
	\abs{p_t-\left(Z_t+\frac{\theta_t}{\alpha}\right)}^2
	-Z_t\theta_t-\frac{\abs{\theta_t}^2}{2\alpha}\\
	&= \frac{\alpha}{2} \abs{p_t - \left(Z_t + \frac{\theta_t}{\alpha} \right)}^2
	-\frac{\alpha}{2} \dis^2\left(Z_t + \frac{\theta_t}{\alpha}, C_t(\mathbb{E}[Z_t]) \right) \geq 0.
\end{align*}
Consequently, after localization, \(R^{(p)}(\cdot,x)\) is a local supermartingale for every \(p\in\scA\). Equality in the drift condition is attained by projecting \(Z_t+\theta_t/\alpha\) onto the constraint set, so the candidate optimal strategy is characterized by
\[
p_t^*\in\Pi_{C_t(\mathbb{E}[Z_t])}\left(Z_t+\frac{\theta_t}{\alpha}\right),
\]
for which \(v(t,p_t^*,Z_t)=0\). Thus, \(R^{(p^*)}(\cdot,x)\) is a local martingale. The existence of the BSDE solution, admissibility of \(p^*\), and the removal of localization are established below in Theorem \ref{thm: app}.

This part identifies the mean-field BSDE and the corresponding pointwise projection rule. It remains to establish well-posedness and complete the verification argument. The utility maximization problem in our setting thereby connects to a class of mean-field quadratic BSDEs (cf.\ \cite{hibon2017MFquadratic}). The verification mechanism follows \cite{hu2005utility}, but the source of the trading constraint is different. In \cite{hu2005utility}, the primitive portfolio constraint is exogenous and fixed. In the present model, the constraint is instead endogenous: \(C_t(m_t)\) depends on the aggregated BSDE signal given consistency condition \(m_t=\bbE[Z_t]\). Consequently, both the admissible set and the BSDE driver are coupled to the solution through the mean-field term. This endogenous dependence is the key extension of the constrained utility framework in \cite{hu2005utility}. To connect the control process \(Z\) in BSDE \eqref{BSDE: application} with the trading strategy \(p\in\scA\), we provide the following remark.

\begin{Remark}\label{Rem: strategy_without_const}
	If the admissible set is taken separately to be the whole space \(C_t=\mathbb{R}^d\), then projection onto the constraint set is the identity:
	\[
	\Pi_{\mathbb{R}^d}(q)=q, \qquad q\in\mathbb{R}^d.
	\]
	Consequently, the projection rule reduces to
	\[
	p_t^* = Z_t + \frac{\theta_t}{\alpha}, \quad \text{a.s., for all } t \in [0, T].
	\]
	Thus, in the unconstrained case, \(Z_t\) is the candidate optimal volatility-scaled position shifted by \(-\theta_t/\alpha\) at time \(t\in[0,T]\). Under constraints, this identity need not hold because the optimal strategy is obtained by projection. In this case, \(\mathbb{E}[Z_t]\) is interpreted as an aggregated BSDE signal that determines the constraint set before the pointwise projection defining \(p_t^*\).
\end{Remark}



\subsection{Existence and uniqueness for BSDEs with aggregated BSDE signal}
In the preceding subsection, we derived the candidate BSDE driver and projection rule under the assumption that BSDE \eqref{BSDE: application} admits a sufficiently integrable solution. In particular, such a solution would give the candidate value function through \(Y_0\) and the candidate optimizer through the projection of \(Z_t+\theta_t/\alpha\). We now justify this formal derivation by establishing a well-posedness result that is strong enough to cover the application.

Rather than working immediately with the specific expectation-dependent driver \(f(t,Z_t,\bbE[Z_t])\) derived in \eqref{eq: driverf}, we first prove a general sufficient result in the copy-space formulation used throughout the paper. Specifically, we consider a driver \(g(t,z,\tilde z)\) satisfying the same \(\Hzero\) and \(\Hqunif\) conditions as in the main theory and study the BSDE
\[
Y_t=\xi+\int_t^T\widetilde{\bbE}\!\left[g(s,Z_s,\widetilde Z_s)\right]ds-\int_t^T Z_s\cdot dW_s.
\]
This copy-space formulation is a technical step rather than a different financial model. Lemma \ref{lem: localauxBSDE} and Proposition \ref{prop: auxBSDE} establish the small-terminal-oscillation result in this general class by local contraction and backward concatenation. Corollary \ref{cor: expectation-dependent} then shows that the same estimates apply to the expectation-dependent BSDE
\[
Y_t=\xi+\int_t^T f(s,Z_s,\bbE[Z_s])\,ds-\int_t^T Z_s\cdot dW_s.
\]
Thus, by taking \(\xi=F\) and \(f\) to be the driver in \eqref{eq: driverf}, the corollary returns exactly to BSDE \eqref{BSDE: application}. The purpose of the results below is therefore to provide a general sufficient well-posedness theorem first, and then use it to complete the utility-maximization verification.

To this end, consider the following copy-space BSDE:
\begin{equation} \label{BSDE_general}
	Y_t=\xi+\int_t^T \int_{\tilde{\Omega}}g(s,Z_s,\tilde Z_s) \diff \tilde{\prob} \, ds -\int_t^T Z_s \cdot \diff W_s, \quad t \in [0, T].
\end{equation}
Set
\[
K:=K_z\vee K_{\tilde z}.
\]
Then $\Hqunif$ implies
\begin{align*}
	&\abs{g(t,z,\tilde z)-g(t,z',\tilde z')}\leq K\bigl(1+|z|+|z'|+|\tilde z|+|\tilde z'|\bigr)
	\bigl(|z-z'|+|\tilde z-\tilde z'|\bigr).
\end{align*}
We begin by establishing the existence and uniqueness of solutions on a sufficiently small time interval in Lemma \ref{lem: localauxBSDE}. By concatenating these local solutions, we extend the result to the global interval $[0, T]$ in Proposition \ref{prop: auxBSDE}. 

First, we note that the terminal condition in BSDE \ref{BSDE: application} is $F$, representing a random liability. It is unrealistic to assume that $F$ is close to zero. Instead, we impose a smallness condition on its oscillation around the mean. Since \(g\) is deterministic, the function \(s\mapsto g(s,0,0)\) is deterministic. We may therefore rewrite BSDE \eqref{BSDE_general} in terms of the centered terminal condition \(\eta:=\xi-\mathbb{E}[\xi]\):
\begin{equation}\label{BSDE_general_auxiliary}
	U_t = \eta + \int_t^T \int_{\tilde{\Omega}} \big[ g(s, Z_s, \tilde{Z}_s) - g(s, 0, 0) \big]\, d\tilde{\mathbb{P}}\, ds - \int_t^T Z_s  \!\cdot dW_s, 
	\qquad 0 \leq t \leq T,
\end{equation}
where
\[
U_t := Y_t-\bbE[\xi]-\int_t^Tg(s,0,0)\,ds
\quad\text{and}\quad
\eta := \xi-\bbE[\xi].
\]
\begin{lemma}\label{lem: localauxBSDE}
	Consider \eqref{BSDE_general_auxiliary} on an interval of length \(h\). Assume that \(g\) satisfies \({\Hzero}\) and \(\Hqunif\), and let \(\zeta\) denote its bounded terminal value on this interval. If
	\[
	h\leq\frac{1}{128K^2}
	\qquad\text{and}\qquad
	a:=\|\zeta\|_{L^\infty}\leq\frac{1}{128K},
	\]
	then the BSDE admits a unique solution in the ball
	\[
	\scB_{4a}:=\left\{(U,Z)\in\bbS^\infty\times\bbH^2_{\mathrm{BMO}}:
	\|U\|_{\bbS^\infty}^2+\|Z\|_{BMO}^2\leq16a^2\right\}.
	\]
\end{lemma}
\begin{proof}
	By a deterministic time translation, it suffices to work on \([0,T]\) with \(T=h\); throughout this proof, we write \(\eta\) for the local terminal value \(\zeta\).
	
	\textbf{Step 1.} We first establish existence and uniqueness for the centered copy-space BSDE \eqref{BSDE_general_auxiliary}.
	For each \((u,z)\in\bbS^\infty\times\bbH^2_{\mathrm{BMO}}\), with \(\tilde z\) a copy of \(z\), define \((U,Z)=\Gamma(u,z)\) to be the solution to
	\[
	U_t
	=
	\eta
	+
	\int_t^T\int_{\tilde\Omega}
	\big[g(s,z_s,\tilde z_s)-g(s,0,0)\big]
	\,d\tilde{\mathbb P}\,ds
	-
	\int_t^T Z_s\cdot dW_s.
	\]
	The standard BMO energy inequalities (see \cite{kazamaki2006BMO}) imply that the frozen driver has sufficient \(L^p\)-integrability for the martingale representation theorem, so the above linear BSDE is well-defined. By \(\Hqunif\) and Jensen's inequality,
	\[
	\left|
	\int_{\tilde\Omega}
	\big[g(s,z_s,\tilde z_s)-g(s,0,0)\big]
	\,d\tilde{\mathbb P}
	\right|
	\leq
	K\bigl(|z_s|+\bbE|z_s|\bigr)
	+
	2K\bigl(|z_s|^2+\bbE|z_s|^2\bigr).
	\]
	For every stopping time \(\tau\leq T\), conditional Cauchy--Schwarz and Jensen's inequality give
	\[
	\bbE_\tau\int_\tau^T|z_s|\,ds
	\leq
	\sqrt T\,\|z\|_{BMO},
	\qquad
	\int_\tau^T\bbE|z_s|\,ds
	\leq
	\sqrt T\,\|z\|_{BMO},
	\]
	and
	\[
	\bbE_\tau\int_\tau^T|z_s|^2\,ds
	\leq
	\|z\|_{BMO}^2,
	\qquad
	\int_\tau^T\bbE|z_s|^2\,ds
	\leq
	\|z\|_{BMO}^2.
	\]
	Consequently,
	\[
	\esssup_{\tau\leq T}
	\left\|
	\bbE_\tau
	\int_\tau^T
	\left|
	\int_{\tilde\Omega}
	\big[g(s,z_s,\tilde z_s)-g(s,0,0)\big]
	\,d\tilde{\mathbb P}
	\right|ds
	\right\|_{L^\infty}
	\leq
	2K\sqrt T\,\|z\|_{BMO}
	+
	4K\|z\|_{BMO}^2.
	\]
	The conditional representation of the linear BSDE and It\^o's formula on \([\tau,T]\) therefore yield
	\[
	\|U\|_{\bbS^\infty}
	\leq
	a
	+
	2K\sqrt T\,\|z\|_{BMO}
	+
	4K\|z\|_{BMO}^2
	\]
	and
	\[
	\|Z\|_{BMO}^2
	\leq
	a^2
	+
	2\|U\|_{\bbS^\infty}
	\left(
	2K\sqrt T\,\|z\|_{BMO}
	+
	4K\|z\|_{BMO}^2
	\right).
	\]
	Set \(R:=4a\) and suppose that \((u,z)\in\scB_R\). Then \(\|z\|_{BMO}\leq R\), while the assumptions of the lemma imply
	\[
	K\sqrt T\leq\frac18,
	\qquad
	KR\leq\frac1{32}.
	\]
	Hence
	\[
	2K\sqrt T\,\|z\|_{BMO}
	+
	4K\|z\|_{BMO}^2
	\leq
	\frac{3R}{8},
	\]
	and therefore
	\[
	\|U\|_{\bbS^\infty}\leq\frac{5R}{8},
	\qquad
	\|Z\|_{BMO}^2\leq\frac{17R^2}{32}.
	\]
	It follows that
	\[
	\|U\|_{\bbS^\infty}^2+\|Z\|_{BMO}^2
	\leq
	\frac{59R^2}{64}
	<
	R^2,
	\]
	so \(\Gamma(\scB_R)\subset\scB_R\).
	
	\textbf{Step 2.} We next prove that \(\Gamma\) is a contraction on \(\scB_R\). Let \((u^i,z^i)\in\scB_R\), \(i=1,2\), and let \((U^i,Z^i)=\Gamma(u^i,z^i)\). By condition \(\Hqunif\), Jensen's inequality, and conditional Cauchy--Schwarz,
	\begin{align*}
		\esssup_{\tau\leq T}
		\left\|
		\bbE_\tau
		\int_\tau^T
		\left|
		\int_{\tilde\Omega}
		\big[
		g(s,z_s^1,\tilde z_s^1)
		-
		g(s,z_s^2,\tilde z_s^2)
		\big]
		\,d\tilde{\mathbb P}
		\right|ds
		\right\|_{L^\infty} &\leq
		\bigl(2K\sqrt T+8KR\bigr)
		\|z^1-z^2\|_{BMO}
		\\
		&
		\leq
		\frac12\|z^1-z^2\|_{BMO}.
	\end{align*}
	The difference BSDE has zero terminal value. Its conditional representation gives
	\[
	\|U^1-U^2\|_{\bbS^\infty}
	\leq
	\frac12\|z^1-z^2\|_{BMO},
	\]
	while It\^o's formula gives
	\[
	\|Z^1-Z^2\|_{BMO}^2
	\leq
	\frac12\|z^1-z^2\|_{BMO}^2.
	\]
	Consequently,
	\[
	\|U^1-U^2\|_{\bbS^\infty}^2
	+
	\|Z^1-Z^2\|_{BMO}^2
	\leq
	\frac34\|z^1-z^2\|_{BMO}^2
	\leq
	\frac34
	\left(
	\|u^1-u^2\|_{\bbS^\infty}^2
	+
	\|z^1-z^2\|_{BMO}^2
	\right).
	\]
	Thus, \(\Gamma\) is a contraction on \(\scB_R=\scB_{4a}\).
	
	\textbf{Step 3.} Finally, by the Banach fixed-point theorem, \(\Gamma\) has a unique fixed point in \(\scB_{4a}\), which is the unique solution of \eqref{BSDE_general_auxiliary} in the stated ball.
\end{proof}

\begin{proposition}\label{prop: auxBSDE}
	Let \(T>0\), and assume that \(g\) satisfies \({\Hzero}\) and \(\Hqunif\). Set
	\[
	\delta:=\frac{1}{128K^2},
	\qquad
	N:=\left\lceil\frac{T}{\delta}\right\rceil,
	\qquad
	\varepsilon_T:=\frac{4^{-(N-1)}}{128K}.
	\]
	If \(\|\eta\|_{L^\infty}\leq\varepsilon_T\), let \(0=t_0<\cdots<t_N=T\) be a deterministic partition with \(t_k-t_{k-1}\leq\delta\), set \(a_N:=\|\eta\|_{L^\infty}\), and define \(a_k:=4^{N-k}a_N\). Define
	\[
	\scB_{\eta,T}
	:=
	\left\{
	(U,Z)\in\bbS^\infty\times\bbH^2_{\mathrm{BMO}}:
	\|U\|_{\bbS^\infty([t_{k-1},t_k])}^2
	+\|Z\|_{\mathrm{BMO}([t_{k-1},t_k])}^2
	\leq a_{k-1}^2,\quad k=1,\ldots,N
	\right\}.
	\]
	Then BSDE \eqref{BSDE_general_auxiliary} admits a unique solution \((U,Z)\in\scB_{\eta,T}\). In particular,
	\[
	\|U\|_{\bbS^\infty}\leq a_0
	\quad\text{and}\quad
	\|Z\|_{\mathrm{BMO}}^2
	\leq\sum_{k=1}^Na_{k-1}^2
	=\frac{16(16^N-1)}{15}\,a_N^2.
	\]
\end{proposition}

\begin{proof}
	We proceed backward from \([t_{N-1},t_N]\). On the last interval, Lemma \ref{lem: localauxBSDE} applies with terminal bound \(a_N\), and the resulting solution belongs to \(\scB_{4a_N}=\scB_{a_{N-1}}\). Inductively, if \(|U_{t_k}|\leq a_k\), the lemma applies on \([t_{k-1},t_k]\), because \(a_k\leq a_1\leq(128K)^{-1}\), and yields the bound \(a_{k-1}=4a_k\). Concatenating the local solutions gives a solution on \([0,T]\). If another solution belongs to \(\scB_{\eta,T}\), backward application of the local uniqueness statement in Lemma \ref{lem: localauxBSDE} shows that it coincides with the constructed solution on each partition interval. Hence the solution is unique in \(\scB_{\eta,T}\).
	
	The \(\bbS^\infty\) estimate is immediate. For a stopping time \(\tau\), conditional expectations over the remaining deterministic subintervals and the tower property give
	\[
	\mathbb{E}_\tau\left[\int_\tau^T|Z_s|^2ds\right]
	\leq\sum_{k=1}^Na_{k-1}^2.
	\]
	Taking the essential supremum over \(\tau\) gives the global BMO estimate.
\end{proof}

\begin{remark}\label{rem:hao-comparison}

The small-terminal-value contraction method of \cite{hao2022solvability} is related to the argument above, but its estimates cannot be applied directly under \(\Hqunif\). Their quadratic Lipschitz condition is homogeneous at the origin. In particular, the relevant driver has the form
\[
|f(x)-f(x')|
\leq
C\bigl(|x|+|x'|\bigr)|x-x'|,
\]
so that shrinking the solution ball also shrinks the contraction coefficient on the whole time interval. In contrast, under \(\Hqunif\) we allow
\[
|g(t,z,\tilde z)-g(t,z',\tilde z')|
\leq
K\bigl(1+|z|+|z'|+|\tilde z|+|\tilde z'|\bigr)
\bigl(|z-z'|+|\tilde z-\tilde z'|\bigr).
\]
The additional constant term gives rise to the contribution \(2K\sqrt h\) in the contraction estimate of Lemma \ref{lem: localauxBSDE}, which is independent of the radius of the solution ball. Hence making the terminal value small does not by itself yield a contraction on the full horizon. Thus, we first choose the time length \(h\) sufficiently small and then concatenate the local solutions backward, leading to the partition-dependent class \(\scB_{\eta,T}\) in Proposition \ref{prop: auxBSDE}.

\end{remark}

Given the solution \((U,Z)\) constructed above for the centered copy-space BSDE \eqref{BSDE_general_auxiliary}, the solution to the copy-space BSDE \eqref{BSDE_general} is obtained by the explicit transformation
\[
Y_t=U_t+\bbE[\xi]+\int_t^Tg(s,0,0)\,ds,\qquad 0\leq t\leq T.
\]
Thus, the smallness condition concerns only the oscillation of \(\xi\) around its mean.
The preceding proposition is stated in the copy-space representation because it matches the general mean-field BSDE framework of the paper. The utility application, however, involves the direct signal \(\bbE[Z_t]\). The following corollary is the bridge between these two formulations: it shows that the same local contraction and concatenation estimates remain valid for a driver of the form \(f(t,Z_t,\bbE[Z_t])\).

\begin{corollary}\label{cor: expectation-dependent}
Let \(f:[0,T]\times\mathbb{R}^d\times\mathbb{R}^d\to\mathbb{R}\) be a deterministic measurable driver such that \(f(\cdot,0,0)\) is bounded and, for some constant \(K>0\),
\[
|f(t,z,m)-f(t,z',m')|
\leq K\bigl(1+|z|+|z'|+|m|+|m'|\bigr)
\bigl(|z-z'|+|m-m'|\bigr).
\]
Define \(\delta,N,\varepsilon_T\) as in Proposition \ref{prop: auxBSDE}, set \(\eta:=\xi-\mathbb{E}[\xi],\) and, for the partition used in Proposition \ref{prop: auxBSDE}, define the centered solution class
\[
\scB^f_{\xi,T}
:=
\left\{
(Y,Z)\in\bbS^\infty\times\bbH^2_{\mathrm{BMO}}:
(U,Z)\in\scB_{\eta,T},
\quad
U_t=Y_t-\mathbb{E}[\xi]-\int_t^T f(s,0,0)\,ds
\right\}.
\]
If \(\|\eta\|_{L^\infty}\leq\varepsilon_T,\) then the expectation-dependent BSDE
\begin{equation}
	Y_t = \xi + \int_t^T f(s, Z_s, \mathbb{E}[Z_s])\, ds
	-\int_t^T Z_s\cdot dW_s,\qquad t\in[0,T],
\end{equation}
admits a unique solution \((Y,Z)\in\scB^f_{\xi,T}\).
\end{corollary}

\begin{proof}

After centering, it is enough to verify that the local self-map and contraction estimates in Lemma \ref{lem: localauxBSDE} remain valid when the second argument is \(\mathbb{E}[z_s]\). By Jensen's inequality,
\[
\big|\mathbb{E}[z_s]\big|\leq\mathbb{E}|z_s|,
\qquad
\big|\mathbb{E}[z_s^1]-\mathbb{E}[z_s^2]\big|
\leq\mathbb{E}|z_s^1-z_s^2|.
\]
Hence the estimates in Steps 1--2 of Lemma \ref{lem: localauxBSDE} hold with the same constants. The backward concatenation argument in Proposition \ref{prop: auxBSDE} therefore applies without change.

\end{proof}

At this point the connection with the financial problem is direct: setting \(\xi=F\) and taking \(f\) from \eqref{eq: driverf} gives precisely BSDE \eqref{BSDE: application}. It remains only to verify that this specific driver satisfies the growth condition of Corollary \ref{cor: expectation-dependent} and then to complete the verification argument.

We now present the main result of this section.

\begin{theorem}\label{thm: app}
Let \(K>0\) be a constant for which the driver \(f\) in \eqref{eq: driverf} satisfies the estimate in Corollary \ref{cor: expectation-dependent}, and define \(\delta,N,\varepsilon_T\) as in Proposition \ref{prop: auxBSDE}. Suppose that the liability \(F\) satisfies
\[
\|F-\mathbb{E}[F]\|_{L^\infty}\leq\varepsilon_T.
\]
The value function of the utility maximization problem is
\[
V(0,x)=-\exp(-\alpha(x-Y_0)),
\]
where \(Y_0\) is defined by the unique solution \((Y,Z)\in\scB^f_{F,T}\) of the  BSDE \eqref{BSDE: application}
\[
Y_t = F + \int_t^T f(s,Z_s, \mathbb{E}[Z_s]) \, ds - \int_t^T Z_s  \!\cdot dW_s, \quad t \in [0, T],
\]
with \(\scB^f_{F,T}\) defined in Corollary \ref{cor: expectation-dependent} and the driver \eqref{eq: driverf}
\[
f(t,z,m)
=
\frac{\alpha}{2}
\dis^2\left(
z+\frac{\theta_t}{\alpha},
C_t(m)
\right)
-
z\theta_t
-
\frac{\abs{\theta_t}^2}{2\alpha}.
\] Moreover, there exists an optimal trading strategy \(p^*\in\scA\) satisfying, almost surely for almost every \(t\in[0,T]\),
\[
p_t^*\in\Pi_{C_t(\mathbb{E}[Z_t])}
\left(Z_t+\frac{\theta_t}{\alpha}\right).
\]
\end{theorem}

\begin{proof}

By Lemma \ref{Lem: conset estimate}, the boundedness of \(\theta\), and \(|a^2-b^2|\leq(a+b)|a-b|\), the driver \(f\) \eqref{eq: driverf} satisfies the growth condition of Corollary \ref{cor: expectation-dependent}, and \(f(\cdot,0,0)\) is bounded. Hence Corollary \ref{cor: expectation-dependent} yields a unique solution
\[
(Y,Z)\in\scB^f_{F,T}\subset\bbS^\infty\times\bbH^2_{\mathrm{BMO}}.
\]
For any \(p\in\scA\), the localized calculation preceding \eqref{eq: driverf} shows that \(R^{(p)}(\cdot,x)\) is a local supermartingale. Since \(Y\in\bbS^\infty\) and \(p\in\scA\),
\[
-R^{(p)}(\tau,x)=e^{-\alpha X_\tau^p}e^{\alpha Y_\tau}
\]
is uniformly integrable over stopping times \(\tau\leq T\). Therefore, \(R^{(p)}(\cdot,x)\) is a supermartingale.
Since \(t\mapsto C_t(\mathbb{E}[Z_t])\) is measurable and compact valued and
\(Z\) is predictable, the projection correspondence
\[
(t,\omega)
\mapsto
\Pi_{C_t(\mathbb{E}[Z_t])}
\left(
Z_t(\omega)+\frac{\theta_t}{\alpha}
\right)
\]
admits a predictable selection \(p^*\). Hence, \(p_t^*
\in
\Pi_{C_t(\mathbb{E}[Z_t])}
\left(
Z_t+\frac{\theta_t}{\alpha}
\right)\)
almost surely for almost every \(t\in[0,T]\), and therefore
\[
v(t,p_t^*,Z_t)=0.
\]
Lemma \ref{Lem: conset estimate} and Jensen's inequality give
\[
|p_t^*|\leq K_1\bigl(1+|\mathbb{E}[Z_t]|\bigr),
\qquad
\int_0^T|\mathbb{E}[Z_t]|^2dt\leq\|Z\|_{\mathrm{BMO}}^2.
\]
Thus \(p^*\cdot W\) and \(Z\cdot W\) are BMO martingales, so
\[
\mathcal E\left(\alpha\int_0^\cdot(Z_s-p_s^*)\cdot dW_s\right)
\]
is a uniformly integrable martingale. Since the drift of \(R^{(p^*)}\) vanishes, \(R^{(p^*)}(\cdot,x)\) is therefore a uniformly integrable martingale. The boundedness of \(Y\) then gives the uniform-integrability, so \(p^*\in\scA\). Finally,
\[
\mathbb{E}\left[-\exp\left(-\alpha(X_T^p-F)\right)\right]
\leq
-\exp\left(-\alpha(x-Y_0)\right)
=
\mathbb{E}\left[-\exp\left(-\alpha(X_T^{p^*}-F)\right)\right],
\]
which proves the value-function identity and the optimality of \(p^*\); see also \cite[Theorem 7]{hu2005utility}.

\end{proof}
\section{Conclusion}
In this paper, we study mean-field BSDEs whose generators have quadratic growth in the control component and in the mean-field term. Our main well-posedness results cover two complementary structural regimes. First, under the \emph{unified} quadratic condition $\Hqunif$, we establish global existence and uniqueness of solutions in $\bbS^\infty\times\bbH^2_{\mathrm{BMO}}$ for bounded terminal conditions of the form $\xi=\phi(W_{t_1},\ldots,W_{t_n})$, where $\phi$ is bounded and H\"older continuous. Second, under the \emph{separably} quadratic condition $\Hqsep$, we prove existence and uniqueness in the same solution class for bounded terminal conditions generated by uniformly continuous functionals of the Brownian path, i.e. $\xi=\phi(W_{[0,T]})$ with $\phi:C([0,T];\bbR^d)\to\bbR$ uniformly continuous in the supremum norm.

Finally, we extend the constrained exponential utility maximization framework of \cite{hu2005utility} to a mean-field setting in which the feasible portfolio set depends endogenously on the aggregated BSDE signal \(m_t=\bbE[Z_t]\). This feedback gives rise to an expectation-dependent quadratic BSDE. Under a small-terminal-oscillation condition on the liability, the associated BSDE admits a unique solution in a suitable small solution class, obtained through a local contraction argument and backward concatenation. The resulting BSDE characterizes the value function and an optimal trading strategy through the projection onto the mean-field adaptive constraint set.
\section*{AI Statement}
The authors acknowledge the use of Grammarly, ChatGPT, and Claude during the preparation of this manuscript to refine the English exposition, assist with the literature review, and check for potential mathematical errors. Specific technical assistance provided by AI includes the following:
\begin{itemize}
	\item In Proposition \ref{prop:qvbound}, after the authors drafted an initial proof containing an error, Claude Opus 4.8 suggested the coupling argument from \cite{song2021gradient} during a verification query. Notably, the model did not initially cite \cite{song2021gradient}, requiring the authors to follow up to trace and properly attribute the original source.
	\item In Theorem \ref{thm: mainthm3}, when prompted for a suitable approximating sequence, Claude Opus 4.8 identified the Moreau--Yosida approximation.
\end{itemize}
All core conceptual developments, main results, and proofs were established by the authors, who independently verified all mathematical arguments and bear full responsibility for the content of the manuscript.

\appendix
\section{Appendix}
\begin{lemma}\label{lem:reflection}
Let $y$ be a positive constant, $B$ be a one-dimensional Brownian motion starting at $0$ and $\tau:=\inf\crl{t>0: B_t=y}$. Then, for $\alpha\in(0,1)$, there exists a constant $C$ which does not depend on $y$ such that
\[
\bbE\edg{|y-B_1|^\alpha\indicator{\tau\geq1}}\leq Cy.
\]
\end{lemma}
\begin{proof}
Note that
\begin{align*}
	\bbE\edg{|y-B_1|^\alpha\indicator{\tau\geq1}}=\bbE\edg{\brak{|y-B_1|\indicator{\tau\geq1}}^\alpha\brak{\indicator{\tau\geq1}}^{1-\alpha}}\leq \brak{\bbE \edg{|y-B_1|\indicator{\tau\geq1}}}^\alpha\brak{\bbE\edg{\indicator{\tau\geq1}}}^{1-\alpha}
\end{align*}
by Holder inequality.
Using the reflection principle (e.g. Theorem 3.7.1 of \cite{shreve2004stochastic}), it is standard to obtain
\[
\bbP(\tau\geq 1)=\frac{2}{\sqrt{2\pi}}\int_{0}^{y}e^{-x^2/2}dx\leq \frac{2y}{\sqrt{2\pi}}.
\]
On the other hand, if we let $R_t=y-B_t$, it is a martingale and
$R_t\indicator{\tau\geq t}= R_{t\wedge \tau}$. By the Optional Stopping Theorem, we have
\begin{align*}
	\bbE \edg{|y-B_1|\indicator{\tau\geq1}}=\bbE \edg{R_1\indicator{\tau\geq1}}=\bbE\edg{R_{1\wedge\tau}}=R_0=y.
\end{align*}
Therefore, we have
\begin{align*}
	\bbE\edg{|y-B_1|^\alpha\indicator{\tau\geq1}}\leq y^\alpha \brak{\frac{2y}{\sqrt{2\pi}}}^{1-\alpha}=\brak{\frac{2}{\sqrt{2\pi}}}^{1-\alpha} y
\end{align*}
This completes the proof.
\end{proof}
\bibliographystyle{abbrvnat}
\bibliography{bib_yining_submission}
\end{document}